\documentclass[10pt]{article}
\usepackage[utf8]{inputenc}
\usepackage{amsmath,verbatim,cite}
\usepackage{amssymb}
\usepackage{graphicx}
\usepackage[affil-it]{authblk}

\usepackage{color}
\usepackage[margin=1.5in]{geometry}
\usepackage{subfigure}

\newtheorem{theorem}{Theorem}[section]

\newenvironment{proof}[1][Proof.]{\begin{trivlist}
\item[\hskip 15pt {\itshape #1}]}{\end{trivlist}}

\newenvironment{remark}[1][Remark.]{\begin{trivlist}
\item[\hskip \labelsep {\bfseries #1}]}{\end{trivlist}}

\newcommand{\qed}{\nobreak \ifvmode \relax \else
      \ifdim\lastskip<1.5em \hskip-\lastskip
      \hskip1.5em plus0em minus0.5em \fi \nobreak
      \vrule height0.75em width0.5em depth0.25em\fi}

\newcommand{\at}{\textup{\text{at}}}
\newcommand{\cg}{\textup{\text{cg}}}
\newcommand{\T}{\textup{\text{T}}}

\newcommand{\R}{\mathbb{R}}

\numberwithin{equation}{section}

\title{Analysis of Transition State Theory Rates upon Spatial Coarse-Graining
\thanks{AB acknowledges support from the Department of Defense (DoD) through
the National Defense Science \& Engineering Graduate Fellowship (NDSEG) Program.
ML was supported in part by the NSF PIRE Grant OISE-0967140, NSF Grant
1310835, AFOSR Award FA9550-12-1-0187, and ARO MURI Award W911NF-14-1-0247.
Work at Los Alamos National Laboratory (LANL) was supported by the
United States Department of Energy, Office of Basic Energy Sciences,
Materials Sciences and Engineering Division. LANL is operated by
Los Alamos National Security, LLC, for the National Nuclear Security
Administration of the U.S. DOE under Contract No. DE-AC52-06NA25396.
During a visit to LANL, AB also received partial support from the
Center for Nonlinear Studies (CNLS) through the Laboratory Directed
Research and Development Program, which paid for mathematical development
in this work.
}}

\author[1]{Andrew Binder}
  \affil[1]{School of Mathematics, 206 Church St. SE \newline University of Minnesota,
         Minneapolis, MN 55455, USA \newline \texttt{bind0090@umn.edu}, \texttt{luskin@umn.edu}}
\author[1]{Mitchell Luskin}
\author[2]{Danny Perez}
 \affil[2]{Theoretical Division T-1, Los Alamos National Laboratory \newline Los Alamos,
        NM 87545, USA \newline \texttt{danny\_perez@lanl.gov},\ \texttt{afv@lanl.gov}}
\author[2]{Arthur F. Voter}

\begin{document}

\maketitle

\begin{abstract}
Spatial multiscale methods have established themselves as useful tools
for extending the length scales accessible by conventional statics (i.e., zero
temperature molecular dynamics).  Recently, extensions of these methods, such as
the finite-temperature quasicontinuum (hot-QC) or Coarse-Grained Molecular
Dynamics (CGMD) methods, have allowed for multiscale molecular
dynamics simulations at finite temperature. Here, we assess the
quality of the long-time
dynamics these methods generate by considering canonical transition rates.
Specifically, we analyze the transition state theory (TST) rates in CGMD
and compare them to the corresponding TST rate of the fully atomistic system.
The ability of such an approach to reliably reproduce the TST rate is verified
through a relative error analysis, which is then used to highlight
the major contributions to the error and guide the choice of degrees
of freedom.  Finally, our analytical results are
compared with numerical simulations for the case of a 1-D chain.
\end{abstract}

\section{Introduction}

Molecular dynamics (MD) ---   the direct integration of atomistic equations of
motion --- provides a powerful tool for the study of chemical and
material processes.  Such an approach accurately captures the
physics at the atomic scale and, in principle, enables the accurate modeling of a wide range of
atomistic systems.  However, despite the high speed of modern computers, MD
simulations still struggle to access the wildly disparate length and time scales
required in many applications.  To partially overcome this difficulty,
multiscale methods that bridge the length-scales from the nano- to the
meso-scales have been proposed. While such methods are
certainly promising, relatively little is known of the effect of
spatial coarse-graining on the quality of the dynamics. Improving our
understanding of these issues is necessary in order to expand upon the range
of problems that can be modeled using such an approach.

We concern ourselves here with spatial multiscale processes for which
the critical atomistic behaviors are localized yet strongly coupled to
the environment through long-range elastic effects. Probably the most well known
numerical method to treat such systems is the quasicontinuum (QC)
method. Specifically, the QC method aims to solve molecular statics (i.e.,
molecular dynamics at zero temperature) problems in such
cases \cite{EBTadmor:1996,BlancLeBrisLegoll2005,acta.atc,bqce12,bqcf.cmame}.
In the QC method, the localized region of interest is treated atomistically in
order to preserve a high degree of accuracy, while the behavior of the
remainder of the system is approximated using continuum
mechanics.  This coupling between the length scales is meant to allow for an
elastic coupling of the two regions, ensuring proper boundary conditions for the
atomistic region.  The number of degrees of freedom necessary to describe the
system is significantly reduced through the use of the Cauchy-Born
approximation and a coarsening of the continuum region via the finite element
method (FEM).  This greatly reduces its computational cost compared to a fully
atomistic solution.

Recently, finite temperature versions of the quasicontinuum method, so-called
hot-QC methods \cite{LMDupuy:2005,EBTadmor:2013}, have been developed
in order to extend the QC approach to finite-temperature molecular dynamics.
Hot-QC was designed to simulate systems held at a constant temperature, which
permits an analysis from a thermodynamic perspective.  Mathematical approaches
to finite temperature equilibrium and dynamics have been given in
\cite{parisfinitetemp10,PPMAK:13,KPS,Blanc201384,0951-7715-23-9-006}.
Hot-QC aims at preserving any thermodynamic quantity that depends only on a
(small) subset of all degrees of freedom.  It has recently been pointed out that
this property implies that transition state theory (TST) rates between
metastable states of the system should be well reproduced insofar as the
system's constituents that are essential to the the transitions are
approximately local to the fully-resolved atomistic region. This property
has been exploited in an extension of these methods --- the hyper-QC
method \cite{hyperqc} --- which seeks to efficiently and accurately simulate
state-to-state dynamics of spatially coarse-grained rare-event systems
through the use of accelerated molecular dynamics \cite{PUSAV2009}.


In this paper, we seek to better understand the error in transition rates
introduced by coarse-graining the periphery of the system.  In order to isolate
this error, we consider the coarse-graining of an atomistic system according to
the coarse-grained molecular dynamics (CGMD) formalism described in
\cite{PhysRevB.72.144104}. However, we note that our choice of coarse variables
differ from that of conventional CGMD, as will be discussed below. CGMD and
hot-QC share the same formal basis, but CGMD provides a closed-form
expression to the coarse Hamiltonian, which enables a mathematical
analysis. Further, it naturally handles the interface between the region to be
treated with atomistic detail and the remainder of the system, in contrast to QC
methods where so-called ghost forces pose additional challenges \cite{acta.atc}.
 In order to obtain closed-form results, we will consider transition rates
computed within the purview of harmonic transition state theory (HTST)
\cite{GHVineyard:1957}. HTST is often the method of choice to approximate
transition rates in hard materials.  Our choice for the dividing surface between
the two metastable regions and how the dividing surface is affected by the
coarse-graining will also be discussed.  The error analysis for the TST rate
will serve as confirmation of the validity of the approach and provide intuition
for the types of error made in the coarsening process for spatial multiscale
methods.

The paper is organized as follows: First, we define a coarse-grained energy in
terms of the atomistic energy to be used in the thermodynamic calculations.
Second, we discuss and analyze the HTST rates in atomistic and
coarse-grained systems and derive the relative error in rates due to
coarse-graining in terms of eigenvalues of the respective
Hamiltonians.  We then discuss how these eigenvalues are affected by
coarse-graining. Following that, we provide numerical
results exhibiting the approximations to the HTST rate made by various
coarse-graining schemes for a 1D system and illustrate the major sources of
error in these computations. We specifically investigate the impact of
the choice of degrees of freedom. Finally, we conclude with general
remarks.

\section{The Coarse-Grained Energy}

Consider a system of $N$ particles in $d$ dimensions held at a fixed temperature
$T$.
Let
$\mathbf{q} \in \R^{dN}$ and $\mathbf{p} \in \R^{dN}$ denote the position and
momentum vectors of the particles respectively.  When necessary, we will denote
the position and momentum vectors of individual particles by $\mathbf{q}_{i}$
and $\mathbf{p}_{i}$ for $1 \leq i \leq N$.  For this paper, we will make use of
mass-weighted coordinates for the position and momentum vectors; that is, we
will consider $\mathbf{\tilde{q}}_{i} = \mathbf{q}_{i}/\sqrt{m_{i}}$, where
$m_{i}$ is the mass of the $i$-th particle so that $\mathbf{\tilde{p}}_{i} =
\mathbf{p}_{i}/\sqrt{m_{i}}$.  However, we will dispense with the tilde notation
and still use $\mathbf{q}$ and $\mathbf{p}$ to denote the mass-weighted
coordinates for position and momentum respectively. The total energy, or
Hamiltonian, of the system will be given by
$\mathcal{H}(\mathbf{q},\mathbf{p})$.  We assume that the Hamiltonian is
separable; that is, we assume that the Hamiltonian may be written as a sum of
the kinetic and potential energies of the system:
\begin{equation*}
 \mathcal{H}(\mathbf{q},\mathbf{p}) = \mathcal{V}(\mathbf{q}) +
\mathcal{K}(\mathbf{p}),
\end{equation*}
where $\mathcal{V}(\mathbf{q})$ denotes the potential energy and
$\mathcal{K}(\mathbf{p})$
denotes the kinetic energy.
%
The total kinetic energy $\mathcal{K}(\mathbf{p})$ is given by
\begin{equation*}
 \mathcal{K}(\mathbf{p}) =
\sum_{i=1}^{N}\frac{1}{2}\|\mathbf{p}_{i}\|^{2},
\end{equation*}
as usual.

In order to coarse-grain the system, we will partition the particles into
representative atoms and constrained atoms. The representative atoms are the
subset of atoms which will be fully resolved in the coarsened system while
the constrained atoms are those atoms which will have their degrees of freedom
removed in the coarse-graining procedure.  We will denote the
partitioning of the position and momentum vectors for the entire system into
representative and constrained components in the following manner:
\begin{equation}\label{Eq:Order}
 \mathbf{q}
 =
 (\mathbf{q}^{r}, \mathbf{q}^{c}),
 \quad
 \mathbf{p} = (\mathbf{p}^{r}, \mathbf{p}^{c}),
\end{equation}
where the superscripts $r$ and $c$ indicate the representative and
constrained components, respectively.  These superscripts will be used
throughout the paper to signal that a given quantity pertains to the
representative atoms or the constrained atoms.  For example, we will let $N^{r}$
and $N^{c}$ denote the number of representative and constrained atoms,
respectively.  Of course, we must then have $N = N^{r} + N^{c}$. Throughout the
paper, we will often simply refer to the representative atoms as repatoms.
Figure \ref{fig:SampleMesh} shows a sample partitioning of particles in a 2D
system.

\begin{figure}[t]
   \centering
   \includegraphics[scale=0.5]{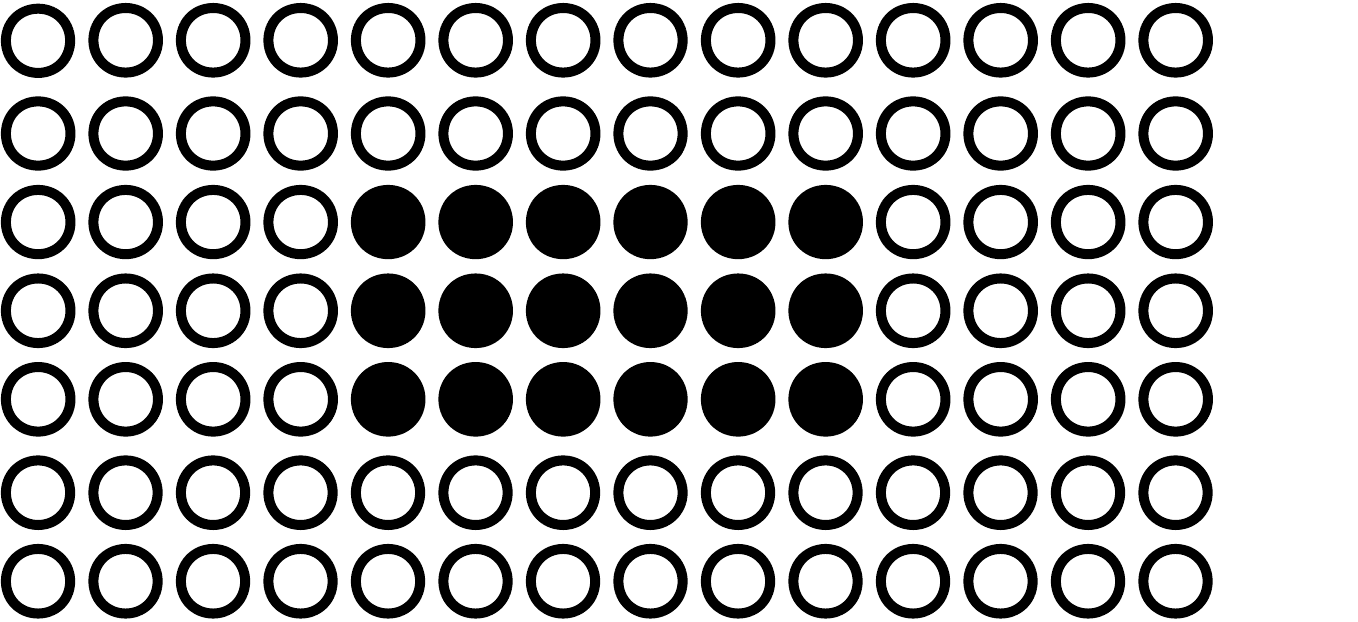}
   \caption{An example of a partitioning of particles in a 2D system.  The
            filled in circles represent the representative atoms or repatoms
            while the empty circles represent the constrained atoms.}
   \label{fig:SampleMesh}
\end{figure}

Note that this is only one of the possible ways to coarsen the variables.
For other choices, the following derivations are still valid as the same block structure of the resolved basis
can be recovered after a simple change of variable.


As we are interested in transitions from one metastable region to another, it
will be useful to consider system properties restricted to a given metastable
region.  Assume that our system initially resides in a metastable region which
we will label as $A$ and that we are interested in transitions
to an adjacent metastable region which we will label $B$.  Let $\Omega_{A}$
denote the set of positions for realizable configurations within the
metastable region $A$ for the system.  In addition, let $\Omega_{A}^{r}$ denote
the set of the positions of the repatoms in these realizable configurations
within $A$; that is, let
\begin{equation}\label{Eq:CoarseDomain}
\Omega_{A}^{r}
:=
\left\{\mathbf{q}^{r} \in \R^{dN^{r}} :
        (\mathbf{q}^{r}, \mathbf{q}^{c}) \in \Omega_{A} \; \text{for some} \;
         \mathbf{q}^{c} \in \R^{dN^{c}}
\right\}.
\end{equation}
We also define $\Omega^{c}_{A}(\mathbf{q}^{r})$ to be the set of constrained
atom positions $\mathbf{q}^{c} \in \R^{dN^{c}}$ such that
$(\mathbf{q}^{r},\mathbf{q}^{c}) \in \Omega_{A}$.

With these newly-defined sets, we define a potential of mean force that we will
take as the potential energy for the coarsened system, as is done in the CGMD, hot-QC,
and hyper-QC methods \cite{PhysRevB.72.144104,LMDupuy:2005,EBTadmor:2013,hyperqc}:
{
\begin{equation*}
 \mathcal{V}^{\cg}(\mathbf{q}^{r}, \beta)
 :=
 -\frac{1}{\beta}\log
 \left(
   \int_{\Omega^{c}_{A}(\mathbf{q}^{r})}
     e^{-\beta\mathcal{V}(\mathbf{q}^{r},\mathbf{q}^{c})}
   d\mathbf{q}^{c}
 \right),
\end{equation*}
where $\beta := (k_{B}T)^{-1}$, $k_{B}$ is Boltzmann's constant, and $T$
is the temperature of the system.  It is important to note that
the coarse-grained energy is dependent upon the temperature of the system.
As mentioned earlier, note that we deviate here from the traditional CGMD method in
the choice of the coarse variables: in the original
formulation, the coarse variables are defined in terms of finite element shape
functions, in contrast to the degrees of freedom of repatoms
\cite{PhysRevB.72.144104}.

This definition of the coarse potential is motivated by the fact that it
preserves thermodynamic properties that are a function of only repatom degrees
of freedom.  Further, this choice also implies the following
equivalence of partition functions for the original, fully-atomistic, and
coarse-grained systems:
\begin{equation}\label{Eq:PartitionFunctionIdentity}
 Z_{\mathcal{V}}
 :=
 \int_{\Omega_{A}}e^{-\beta\mathcal{V}(\mathbf{q})}d\mathbf{q}
 =
 \int_{\Omega^{r}_{A}}e^{-\beta\mathcal{V}^{\cg}(\mathbf{q}^{r},\beta)}d\mathbf{
q}^{r}
 =:
 Z^{\cg}_{\mathcal{V}},
\end{equation}
where $Z_{\mathcal{V}}$ and $Z^{\cg}_{\mathcal{V}}$ are the elements
of the total partition functions pertaining to the potential energy for the
original and coarsened systems, respectively.


We similarly define the coarse-grained kinetic energy to
be an effective kinetic energy, which may be computed analytically:
\begin{equation*}
 \mathcal{K}^{\cg}(\mathbf{p}^{r},\beta)
 =
 -\frac{1}{\beta}\log
   \left(
     \int
       e^{-\beta\mathcal{K}(\mathbf{p}^{r},\mathbf{p}^{c})}
     d\mathbf{p}^{c}
   \right)
 =
 \sum_{i=1}^{N^{r}}
   \frac{1}{2}\|\mathbf{p}_{i}^{r}\|^{2}
 -
 \frac{d}{2\beta}\sum_{i=1}^{N^{c}}\log
   \left(\frac{2\pi}{\beta}\right).
\end{equation*}
%
Again, this choice gives consistent thermodynamics for quantities
involving only repatom degrees of freedom, and it also yields equal
kinetic energy partition functions for the original and coarsened systems.
 With the kinetic energy thus defined, the total energy or Hamiltonian of the
coarse-grained system is defined to be the sum of the two coarse-grained
energies: $\mathcal{H}^{\cg}(\mathbf{q}^{r},\mathbf{p}^{r},\beta) :=
\mathcal{V}^{\cg}(\mathbf{q}^{r},\beta) +
\mathcal{K}^{\cg}(\mathbf{p}^{r},\beta)$. This Hamiltonian can then be used to
carry out molecular dynamics simulations.

\section{Transition State Theory (TST) Rate}

We are interested in estimating the rate at which a system residing in the
metastable region $A$ crosses over into the metastable region
$B$. True transition rates are in general difficult to compute directly. A
common approximation to the true transition rate is given by
transition state theory. In TST, one assumes that that once the system crosses the (hyper-)surface
between states $A$ and $B$ --- the so-called dividing surface --- it
will thermalize in state $B$; i.e., it assumes that the trajectory won't rapidly
cross back to $A$ or leave to another state $C$ before losing its
memory in $B$. This assumption is almost never exactly realized,
but nevertheless, it is often an excellent approximation.
With these assumptions, we may define the TST rate for
the fully atomistic system to be the equilibrium flux across the
dividing surface $\Gamma_{AB}$. In the canonical ensemble (NVT), the
rate becomes:
\begin{equation*}
R^{\text{TST}}_{A \to B}
  =
\frac{\frac{1}{2}
  \int\int_{\Gamma_{AB}}|\mathbf{p}\cdot\mathbf{n}|
    e^{-\beta\mathcal{H}(\mathbf{q},\mathbf{p})}
  dSd\mathbf{p}
     }
     {\int\int_{\Omega_{A}}
       e^{-\beta\mathcal{H}(\mathbf{q},\mathbf{p})}
      d\mathbf{q}d\mathbf{p}
     }
  =
  \frac{1}{2}\sqrt{\frac{2}{\pi\beta}}
  \frac{\int_{\Gamma_{AB}}
          e^{-\beta\mathcal{V}_{s}(\mathbf{q})}
        dS}
       {\int_{\Omega_{A}}
         e^{-\beta\mathcal{V}(\mathbf{q})}
        d\mathbf{q}
       },
\end{equation*}
where $\mathbf{n}$ is the vector normal to the dividing
surface and $dS$ indicates that the integration with respect to position is
taken over the surface $\Gamma_{AB}$.  As we will be treating the potential
energy at the dividing surface separately from the potential energy in state
$A$, for clarity, we denote the potential energy at the dividing surface using
$\mathcal{V}_{s}$. This notation will be used throughout the rest of the paper.
The remaining variables are as they have been defined in previous sections.
Note that in the above equation we are able to integrate the momentum portion of
the integral as our dividing surface is taken to be a hyperplane and the
momentum integral is carried out over $\R^{dN}$ \cite{TSTVoter}. Let us define
the partition function
\begin{equation*}
 Z^{\neq}_{\mathcal{V}}
 :=
 \int_{\Gamma_{AB}}e^{-\beta\mathcal{V}_{s}(\mathbf{q})}dS
 \quad  \text{and recall that} \quad
 Z_{\mathcal{V}}
 =
 \int_{\Omega_{A}}e^{-\beta\mathcal{V}(\mathbf{q})}d\mathbf{q}
\end{equation*}
so that we may write the TST rate in the following way:
\begin{equation}\label{Eq:TSTRate}
 R^{\text{TST}}_{A \to B}
 =
 \frac{1}{2}\sqrt{\frac{2}{\pi\beta}}
 \frac{Z^{\neq}_{\mathcal{V}}}{Z_{\mathcal{V}}}.
\end{equation}

Analogously, we define the TST rate in the coarse-grained system as:
\begin{equation*}
 R^{\text{cg}}_{A \to B}
  =
  \frac{
  \frac{1}{2}
    \int\int_{\Gamma^{\text{cg}}_{AB}}
      |\mathbf{p}\cdot\mathbf{n}|
      e^{-\beta\mathcal{H}^{\text{cg}}(\mathbf{q}^{r},\mathbf{p}^{r},\beta)}
    dS^{r}d\mathbf{p}^{r}
       }
       {
         \int\int_{\Omega^{r}_{A}}
          e^{-\beta\mathcal{H}^{\text{cg}}(\mathbf{q}^{r},\mathbf{p}^{r},\beta)}
         d\mathbf{q}^{r}d\mathbf{p}^{r}
       }
  =
  \frac{1}{2}\sqrt{\frac{2}{\pi\beta}}
  \frac{
    \int_{\Gamma^{\text{cg}}_{AB}}
      e^{-\beta\mathcal{V}^{\text{cg}}_{s}(\mathbf{q}^{r},\beta)}
    dS^{r}
       }
       {
        \int_{\Omega^{r}_{A}}
         e^{-\beta\mathcal{V}^{\text{cg}}(\mathbf{q}^{r},\beta)}
        d\mathbf{q}^{r}
       },
\end{equation*}
where $\mathbf{n}$ is the vector normal to the coarse-grained dividing surface,
and $dS^{r}$ indicates that the integral for the position of the atoms is taken
over the corresponding coarse-grained dividing surface. The superscript $r$
serves as a reminder that this integral involves integrating over only the
repatom positions.  Again, as we will be treating the potential energy at the
dividing surface separately from the potential energy in $\Omega^{r}_{A}$, for
clarity, we denote the potential energy at the dividing surface using
$\mathcal{V}^{\text{cg}}_{s}$. Let us define
\begin{equation*}
 Z^{\cg,\neq}_{\mathcal{V}}
 :=
 \int_{\Gamma^{\text{cg}}_{AB}}
   e^{-\beta\mathcal{V}^{\text{cg}}_{s}(\mathbf{q}^{r},\beta)}
 dS^{r}
 \quad \text{and recall that} \quad
 Z^{\text{cg}}_{\mathcal{V}}
 =
 \int_{\Omega^{r}_{A}}
   e^{-\beta\mathcal{V}^{\text{cg}}(\mathbf{q}^{r},\beta)}
 d\mathbf{q}^{r}
\end{equation*}
so that we may write the coarse-grained TST rate as
\begin{equation}\label{Eq:CGTSTRate}
 R^{\text{cg}}_{A \to B}
 =
 \frac{1}{2}\sqrt{\frac{2}{\pi\beta}}
 \frac{Z^{\text{cg},\neq}_{\mathcal{V}}}{Z^{\text{cg}}_{\mathcal{V}}}.
\end{equation}

While formally simple, the TST approximation to the transition rate
usually does not allow for closed-form results because the partition function
integrals cannot be carried out for general potentials. This
difficulty is compounded by the need to integrate along a potentially
complex dividing surface. These two challenges can be addressed through  the
so-called Harmonic approximation to TST (HTST)
\cite{GHVineyard:1957}. HTST introduces two additional assumptions:
i) the kinetic bottleneck for the transition corresponds to crossing
an energy barrier (culminating at a first order saddle point) that stands between $A$ and $B$, and ii) for the
purpose of the calculation of the partition functions, the potential
can locally be expanded to second order. These properties will be used to give
an explicit definition of a dividing surface and to analytically compute the
partition functions entering into the rate expression using this surface.  The
HTST assumptions are particularly appropriate when states correspond to basins
of attraction of a single minimum on the potential energy surface
(which is often the case for solid-state kinetics) and when the temperature is sufficiently low.

Consider the saddle point $\mathbf{q}_{\text{s}}$ connecting the
states $A$ and $B$.  The potential energy around  $\mathbf{q}_{\text{s}}$ is
then approximated as:
\begin{equation}\label{Eq:Saddle}
 \mathcal{V}_{s}(\mathbf{q})
 :=
 \mathcal{V}(\mathbf{q}_{\text{s}})
 +
 \frac{1}{2}{\mathbf{u}}\cdot{\mathbf{D}}^{\at}{\mathbf{u}},
 \qquad
 {\mathbf{u}}
 :=
 \mathbf{q} - \mathbf{q}_{s},
\end{equation}
where ${\mathbf{u}} := \mathbf{q} -
\mathbf{q}_{s}$ is the vector displacement of all of the atoms from their
positions at the saddle point, and ${\mathbf{D}}^{\at}$ is the Hessian
matrix evaluated at the saddle point.  Explicitly,
\begin{equation*}
 {\mathbf{D}}^{\at}_{ij} :=
 \frac{\partial^{2}\mathcal{V}}
      {\partial \mathbf{q}_{i}\partial\mathbf{q}_{j}}(\mathbf{q}_{\text{s}}),
 \quad\quad 1 \leq i,j \leq N.
\end{equation*}
Note that the above approximation for the potential energy at the saddle point
will be used in conjunction with the other HTST assumption in the computation
of \eqref{Eq:TSTRate}, specifically the computations involving the dividing
surface. A corresponding expansion could be carried out around the potential
energy minimum $\mathbf{q}_{\text{m}}$ in state $A$ in terms of a different
Hessian matrix $\mathbf{D}^{\at}_{\mathrm{m}}$, but this is not necessary with
our formulation of the problem.

At a first-order saddle point, the Hessian has one negative eigenvalue
while the rest are positive (assuming the absence of free translations or
rotations). This offers a natural definition of the dividing surface
as the hyperplane that passes through $\mathbf{q}_{s}$ and whose normal vector
is the unstable eigenmode of the system's dynamical matrix at this point. For
the remainder of the paper, $\Gamma_{AB}$ will be used to denote the dividing
surface defined by these conditions.  This choice conveniently allows for the
explicit calculation of the saddle point partition function as will be shown
below.

We will similarly provide an explicit definition for the dividing surface
$\Gamma^{\cg}_{AB}$ in the coarsened phase space.  This requires that we first
determine the appropriate saddle point in the coarse-grained phase space and its
corresponding dynamical matrix.  As we are projecting the fully atomistic
system into a repatom subspace of this system, we might expect
$\mathbf{q}^{r}_{s}$, the repatom components of the saddle point, to be the
transition state in the coarse-grained space. We are interested, then, in
verifying whether this is actually the case and computing the associated
coarse-grained dynamical matrix.

To begin the derivation of the coarse-grained saddle point,
consider the harmonic approximation of  \eqref{Eq:Saddle}.
By ordering the position variable $\mathbf{q} = (\mathbf{q}^{r},
\mathbf{q}^{c})$ following  \eqref{Eq:Order}, the Hessian matrix
has the block-form structure
\begin{equation}\label{Eq:AtDynamicalMatrix}
\begin{split} \mathbf{D}^{\at} :=
   \begin{pmatrix}
    \mathbf{R} & \mathbf{B} \\
    \mathbf{B}^{\T} & \mathbf{C}
   \end{pmatrix},
\end{split}\end{equation}
where
\begin{equation*}\begin{split}
 \mathbf{R}_{ij} = \frac{\partial^{2}\mathcal{V}_{s}}
 {\partial\mathbf{q}^{r}_{i}\partial \mathbf{q}^{r}_{j}}(\mathbf{q}_{\text{s}}),
 \;\; 1 \leq i,j \leq N^{r};
 \quad
 \mathbf{C}_{k\ell} = \frac{\partial^{2}\mathcal{V}_{s}}
 {\partial\mathbf{q}^{c}_{k}\partial\mathbf{q}^{c}_{\ell}}
 (\mathbf{q}_{\text{s}}), \;\; 1 \leq k,\ell \leq N^{c};\\
%
 \mathbf{B}_{mn} = \frac{\partial^{2}\mathcal{V}_{s}}
 {\partial\mathbf{q}^{r}_{m}\partial \mathbf{q}^{c}_{n}}(\mathbf{q}_{\text{s}}),
 \quad 1 \leq m \leq N^{r}, 1 \leq n \leq N^{c}.
\end{split}\end{equation*}
%
%
Now, we may define a coarse-grained energy near the saddle point with the
domain given by the subspace in  \eqref{Eq:CoarseDomain}:
\begin{equation}\label{Eq:CGFreeEnergy}
 \mathcal{V}^{\cg}_{s}(\mathbf{q}^{r},\beta) :=
 -\frac{1}{\beta}\log\left(
   \int_{\R^{dN^{c}}}
     e^{-\beta\mathcal{V}_{s}(\mathbf{q}^{r},\mathbf{q}^{c})}d\mathbf{q}^{c}
   \right).
\end{equation}
The integration in the above definition is taken over all of $\R^{dN^{c}}$
rather than $\Omega_{A}^{c}(\mathbf{q}^{r})$ to allow for a closed form
expression and is considered to be part of the assumptions for HTST in regards
to treating the potential energy as second-order.  We may compute the integral
directly using  \eqref{Eq:Saddle}:
\begin{equation}\label{Eq:CGDividingSurfaceEnergy}
\begin{split}
 \mathcal{V}^{\cg}_{s}(\mathbf{q}^{r},\beta)
 &=
 \mathcal{V}(\mathbf{q}_{s}) - \frac{1}{\beta}
 \log
 \left(
   \int
     e^{-\frac{\beta}{2}\mathbf{u}\cdot\mathbf{D}^{\at}\mathbf{u}}
   d\mathbf{q}^{c}
 \right)
 \\ &=
 \mathcal{V}(\mathbf{q}_{s}) -\frac{1}{\beta}
 \log
   \left(
     \int
       e^{-\frac{\beta}{2}
         \left(
           (\mathbf{u}^{c} - \mathbf{u}^{c}_{\text{min}})
           \cdot \mathbf{C}
           (\mathbf{u}^{c} - \mathbf{u}^{c}_{\text{min}})
         +
           \mathbf{u}^{r} \cdot \mathbf{D}^{\cg} \mathbf{u}^{r}
         \right)
         }
     d\mathbf{q}^{c}\right)
 \\ &=
 \mathcal{V}(\mathbf{q}_{s})
   + \frac{1}{2\beta}
     \log
       \left(
         \frac{\det\mathbf{C}}{(2\pi/\beta)^{dN^{c}}}
       \right)
    + \frac{1}{2}
      \mathbf{u}^{r} \cdot \mathbf{D}^{\cg} \mathbf{u}^{r},
\end{split}
\end{equation}
where
\begin{equation}\label{Eq:RelaxedConstrainedAtoms}
  \mathbf{u}^{c}_{\text{min}} := -\mathbf{C}^{-1}\mathbf{B}^{\T}\mathbf{u}^{r},
  \quad
  \mathbf{D}^{\cg} := \mathbf{R} - \mathbf{B}\mathbf{C}^{-1}\mathbf{B}^{\T},
\end{equation}
and we have assumed that the matrix $\mathbf{C}$ is invertible and
positive-definite.  From this, we can see that $\mathbf{q}^{r}_{s}$
(equivalently, $\mathbf{u}^{r} = \mathbf{0}$) is a saddle point of the
coarse-grained system with its corresponding dynamical matrix being
\begin{equation}\label{Eq:CGDynMatrix}
 \mathbf{D}^{\cg} = \mathbf{R} - \mathbf{B}\mathbf{C}^{-1}\mathbf{B}^{\T}
\end{equation}
provided that $\mathbf{D}^{\cg}$ has both positive and negative eigenvalues.
In order for the dividing surface $\Gamma^{\cg}_{AB}$ to be well-defined, recall
that the matrix $\mathbf{D}^{\cg}$ must have only one negative eigenvalue and
that the rest of its eigenvalues must be positive.  We
will elaborate on these requirements for $\mathbf{C}$ and
$\mathbf{D}^{\cg}$ and under what circumstances they can be guaranteed
to be met shortly. Before that, observe that the relation
\begin{equation}\label{Eq:QuadraticPortionofAtomisticEnergy}
 \mathbf{u}\cdot\mathbf{D}^{\at}\mathbf{u} =
 (\mathbf{u}^{c}- \mathbf{u}^{c}_{\text{min}})\cdot
 \mathbf{C}(\mathbf{u}^{c} - \mathbf{u}^{c}_{\text{min}})
 +
 \mathbf{u}^{r}\cdot\mathbf{D}^{\cg}\mathbf{u}^{r}
\end{equation}
arises from multiplying the displacement vector and dynamical matrix in their
partitioned forms from  \eqref{Eq:AtDynamicalMatrix} and then completing
the square for the constrained components.  Writing this portion of the
atomistic energy in this form makes it clear that for a given $\mathbf{u}^{r}$,
$\mathbf{u}^{c}_{\text{min}}$ gives the energy-minimizing displacements for the
constrained atoms, thus motivating the choice of notation for this vector.
This quantity will play an important role in the discussion on the error in the
coarse-grained approximation of the TST rate.
An equivalent
derivation can be carried out around the energy minimum. However, as
this quantity will not be used in the following, the derivation is omitted.

Note that, in order to facilitate the formal analysis, the method we consider here does not exactly correspond
to either the CGMD or hot-QC methods. As mentioned earlier, degrees of freedom
in CGMD are usually defined in terms of finite element shape functions and not
in terms of repatoms.  Further, the coarse-grained Hamiltonian is computed only once from
 \eqref{Eq:CGFreeEnergy} using a harmonic approximation around a specific
value of $\mathbf{q}^{r}$ (usually corresponding to the energy
minimum). In the case of hot-QC, additional approximations intervene
in the calculation of the coarse-grained Hamiltonian, namely, the
harmonic approximation is replaced by a local-harmonic approximation
and  the integral over the constrained atoms
is further approximated using the finite element method and a Cauchy-Born
approximation based on a set of nodes (which are distinct from repatoms) placed
in the periphery. In the ``static'' variant of hot-QC, the displacement of the
nodes is chosen to minimize an approximation of the (free-)energy of the
constrained atoms with respect to the instantaneous $\mathbf{q}^{r}$ while in
the ``dynamic'' variant, the nodes are allowed to move dynamically in
order to reduce the computational cost inherent to the minimization.

Our discussion pertains to a hypothetical
method that
combines the best of CGMD and hot-QC, i.e., where coarse-graining is carried
out {\em exactly} at the harmonic level with respect to the instantaneous
$\mathbf{q}^{r}$. At sufficiently low temperature, the error of such a method is therefore
dominated by the coarse-graining error and provides a
lower bound on the error in an actual CGMD or hot-QC model.
A complete analysis of the rate errors in hot-QC would have to
consider all of the additional approximations, which is beyond the
scope of the current paper. However, in such an analysis, the error
contributed solely by the coarse-graining process would exactly correspond to what will be
derived below.

Returning to the properties of the matrices $\mathbf{C}$ and
$\mathbf{D}^{\cg}$, we first note that whether $\mathbf{C}$ is invertible and
$\mathbf{D}^{\cg}$ has a negative eigenvalue is entirely dependent upon a
sensible choice of the repatom region for the problem.  Provided that the
essential transition behavior is contained within the chosen repatom region, we
expect that the matrices $\mathbf{C}$ and $\mathbf{D}^{\cg}$ will satisfy these
conditions.  Assuming this to be the case, it can be shown that $\mathbf{C}$ is
also positive definite and that the eigenvalues of $\mathbf{D}^{\cg}$ are such
that $\Gamma^{\cg}_{AB}$ is well defined.  We now state
without proof a version of Cauchy's Interlacing Theorem to be used in the major
theorem of this section proving the preceding statements.
\begin{theorem}
[Cauchy's Interlacing Theorem]
Let $\mathbf{S}$ be a symmetric $n \times n$
matrix.  Define the orthogonal projection matrix $\mathbf{P}$ in block form to be
$\mathbf{P} :=
 \begin{pmatrix}
  \mathbf{I}_{m} & \mathbf{0} \\
  \mathbf{0} & \mathbf{0}
 \end{pmatrix},
$
where $\mathbf{I}_{m}$ is an $m \times m$ identity matrix with $m < n$ and the
remainder of the blocks are zero matrices of the appropriate dimensions.  Let
$\mathbf{T}$ denote the upper left $m \times m$ matrix block of
$\mathbf{P}^{\T}\mathbf{S}\mathbf{P}$:
\begin{equation*}
 \mathbf{P}^{\T}\mathbf{S}\mathbf{P} =
   \begin{pmatrix}
    \mathbf{T} & \mathbf{0} \\
    \mathbf{0} & \mathbf{0}
   \end{pmatrix},
\end{equation*}
where the remainder of the blocks are zero matrices of the appropriate
dimensions. If the eigenvalues of $\mathbf{S}$
are $\sigma_{1} \leq \sigma_{2} \leq \cdots \leq \sigma_{n}$ and the eigenvalues
of $\mathbf{T}$ are $\tau_{1} \leq \tau_{2} \leq \cdots \leq \tau_{m}$,
then
\begin{equation}\label{Eq:InterlacingEigenvalues}
 \sigma_{j} \leq \tau_{j} \leq \sigma_{n-m+j}, \;\; \text{for} \;\; 1 \leq j
\leq m.
\end{equation}
\end{theorem}
\begin{proof}
See \cite{KatoPerturbation}.  This may be proved using Courant's Min-Max
Theorem.  
\end{proof}
Now, we prove our claim.
\begin{theorem}\label{Thm:WellBehaved}
Let $\mathbf{D}^{\at}$ and $\mathbf{D}^{\cg}$
be the fully atomistic dynamical matrix and the coarse-grained dynamical
matrix as defined in  \eqref{Eq:AtDynamicalMatrix}
and \eqref{Eq:CGDynMatrix}, respectively.  Recall that we assume that
$\mathbf{D}^{\at}$ has only one negative eigenvalue while the rest are
positive.  For ease of reference,
\begin{equation}\label{Eq:WellBehavedThmRefEq}
 \mathbf{D}^{\at} =
   \begin{pmatrix}
    \mathbf{R} & \mathbf{B} \\
    \mathbf{B}^{\T} & \mathbf{C}
   \end{pmatrix}
 \;\; \text{and} \;\;
 \mathbf{D}^{\cg} = \mathbf{R} - \mathbf{B}\mathbf{C}^{-1}\mathbf{B}^{\T}.
\end{equation}
We also require that the matrix $\mathbf{C}$ be non-singular so that
the above definitions make sense.  If $\mathbf{D}^{\cg}$ has a negative
eigenvalue, then the remaining eigenvalues of $\mathbf{D}^{\cg}$ are
positive and the matrix $\mathbf{C}$ is positive definite.  In addition, the
negative eigenvalue of $\mathbf{D}^{\cg}$ is greater than or equal to in
absolute value that of $\mathbf{D}^{\at}$.  In symbols,
\begin{equation*}
  |\lambda^{\cg}| \geq |\lambda^{\at}|.
\end{equation*}
\end{theorem}
\begin{proof}
As the matrix $\mathbf{D}^{\at}$ possesses no zero eigenvalue, it
is invertible. Since the matrix $\mathbf{C}$ is also invertible, we
may use a standard block-matrix determinant identity and
\eqref{Eq:WellBehavedThmRefEq} to show the following:
\begin{equation}\label{Eq:BlockDeterminantIdentity}
 \det\mathbf{D}^{\at} = \det\mathbf{C}\det\mathbf {D}^{\cg}.
\end{equation}
Thus, the determinant of $\mathbf{D}^{\cg}$ is non-zero, so $\mathbf{D}^{\cg}$
is non-singular.  We may then compute the inverse of $\mathbf{D}^{\at}$ in block
form which is
\begin{equation*}
 (\mathbf{D}^{\at})^{-1} =
 \begin{pmatrix}
  (\mathbf{D}^{\cg})^{-1} & -(\mathbf{D}^{\cg})^{-1}\mathbf{B}\mathbf{C}^{-1} \\
  -\mathbf{C}^{-1}\mathbf{B}^{\T}(\mathbf{D}^{\cg})^{-1} & \mathbf{C}^{-1} +
\mathbf{C}^{-1}\mathbf{B}^{\T}(\mathbf{D}^{\cg})^{-1}\mathbf{B}\mathbf{C}^{-1}
 \end{pmatrix}.
\end{equation*}
Note that, as the eigenvalues of $(\mathbf{D}^{\at})^{-1}$ are the
multiplicative inverses of the eigenvalues of $\mathbf{D}^{\at},$
$(\mathbf{D}^{\at})^{-1}$ has one negative eigenvalue with the rest being
positive.  We now apply Cauchy's Interlacing Theorem to the matrices
$(\mathbf{D}^{\at})^{-1}$ and $(\mathbf{D}^{\cg})^{-1}$ to place bounds on the
eigenvalues of $(\mathbf{D}^{\cg})^{-1}$.  Let $\lambda^{\at}$ denote the
single negative eigenvalue of $\mathbf{D}^{\at}$.  It is important to note that
$(\lambda^{\at})^{-1}$
is less than all of the other eigenvalues of $(\mathbf{D}^{\at})^{-1}$ due to
its sign.  Since we are given that $\mathbf{D}^{\cg}$ possesses
a negative eigenvalue and thus that $(\mathbf{D}^{\cg})^{-1}$ possesses
a negative eigenvalue, Cauchy's Interlacing Theorem immediately implies that
the remaining eigenvalues of $(\mathbf{D}^{\cg})^{-1}$ are positive by
\eqref{Eq:InterlacingEigenvalues}.  Let $(\lambda^{\cg})^{-1}$ denote the
single negative eigenvalue of $(\mathbf{D}^{\cg})^{-1}$.  Cauchy's Interlacing
Theorem implies that $(\lambda^{\at})^{-1} \leq (\lambda^{\cg})^{-1}$.  Keeping
in mind that these eigenvalues are negative, this inequality implies the
result
\begin{equation}\label{Eq:CurvatureComparison}
 |\lambda^{\cg}| \geq |\lambda^{\at}|.
\end{equation}
Inverting the eigenvalues of $(\mathbf{D}^{\cg})^{-1}$ to arrive at the
eigenvalues of $\mathbf{D}^{\cg}$ does not change their sign, so we have
finished the proof that the spectrum of $\mathbf{D}^{\cg}$ has the properties
as claimed in the statement of the theorem.

In order to finish the proof, observe that
\eqref{Eq:BlockDeterminantIdentity} implies that the determinant
of $\mathbf{C}$ is positive as the determinant of $\mathbf{D}^{\at}$ and
$\mathbf{D}^{\cg}$ are both negative. Cauchy's Interlacing Theorem applied to
$\mathbf{D}^{\at}$ and $\mathbf{C}$ implies that $\mathbf{C}$ may have at most
one negative eigenvalue according to  \eqref{Eq:InterlacingEigenvalues}.  As
having just one negative eigenvalue would force the determinant of $\mathbf{C}$
to be negative and contradict our determinant identity, we must have that all of
the eigenvalues of $\mathbf{C}$ are in fact positive. 
\end{proof}

This theorem proves that $\mathbf{q}^{r}_{s}$ is indeed a saddle point, assuming the repatom region to be appropriately selected.
In particular, it shows that no additional transition pathways are introduced in
the coarsened system and that if the coarsened system has a transition
pathway, it uniquely corresponds to a transition pathway in the fully
atomistic system. Most interestingly, this result implies
that coarse-graining the system {\em never} decreases the absolute curvature of
the actual transition pathway at the barrier in the potential energy surface due
to the inequality in the negative eigenvalues of the dynamical
matrices.  Equivalently, this implies that the magnitude of the imaginary
eigenmode frequency for the coarsened system is never less than the magnitude of
the imaginary eigenmode frequency in the fully atomistic system.  This will have
implications for the TST rate in the coarsened system that will be made
clear over the course of the analysis of the TST rate approximation in the next
section.

Finally, we may now provide our explicit definition of the dividing surface
$\Gamma^{\cg}_{AB}$ in the coarse-grained system assuming that the repatom
region was appropriately chosen.  We define this surface to be the hyperplane
that passes through the saddle point $\mathbf{q}^{r}_{s}$ and has as its normal
vector the unstable eigenmode of $\mathbf{D}^{\cg}$.  The dynamical matrix
$\mathbf{D}^{\cg}$ takes the form of a repatom Hessian with a correction due to
the interaction between the representative and constrained atoms.  This
correction can easily be understood from a physical point of view after
considering the implications of  \eqref{Eq:RelaxedConstrainedAtoms} and
\eqref{Eq:QuadraticPortionofAtomisticEnergy}.  Given a displacement of the
repatoms $\mathbf{u}^{r}$, the matrix $-\mathbf{C}^{-1}\mathbf{B}^{\T}$ in the
definition of $\mathbf{D}^{\cg}$ applied to $\mathbf{u}^{r}$ will give the
displacement of the constrained atoms that will yield a minimized energy for
the fully atomistic system in the context of HTST.  In other words,
$-\mathbf{C}^{-1}\mathbf{B}^{\T}$ finds the relaxed constrained atom
configuration for the problem.  The application of the $\mathbf{B}$ matrix is
necessary for determining the force such a configuration would then exert on the
repatoms.  Thus, the correction to the repatom Hessian is due to the constrained
atoms in their relaxed state. 

\section{TST Rate Error Analysis}

With this dividing surface now defined, we
may analyze the error in the TST rate made by coarse-graining the system.
For this, an absolute error analysis is not useful as most classical
transition rates vanish in the zero temperature
limit.  Therefore, convergence is essentially already guaranteed by
the exponentiation of the energy in the Gibbs measure.  In order to conduct a more
meaningful analysis, the relative error in the TST rate will be examined
instead. Using the definitions from \eqref{Eq:TSTRate} and
\eqref{Eq:CGTSTRate}, the relative error is seen to be
\begin{equation}\label{Eq:RelativeError}
 \left|\frac{R^{TST}_{A \to B}  - R^{\text{cg}}_{A \to B}}
            {R^{TST}_{A \to B}}\right|
 =
 \left|1 -
 \frac{Z_{\mathcal{V}}}{Z^{\text{cg}}_{\mathcal{V}}}
 \frac{Z^{\text{cg},\neq}_{\mathcal{V}}}{Z^{\neq}_{\mathcal{V}}}
 \right|.
\end{equation}
This relative error can be computed by calculating the two ratios
$Z_{\mathcal{V}}/Z^{\text{cg}}_{\mathcal{V}}$ and
$Z^{\text{cg},\neq}_{\mathcal{V}}/Z^{\neq}_{\mathcal{V}}$ separately.  The
first ratio $Z_{\mathcal{V}}/Z^{\text{cg}}_{\mathcal{V}}$ is trivial: as was shown in  \eqref{Eq:PartitionFunctionIdentity},
this ratio is simply $Z_{\mathcal{V}}/Z^{\text{cg}}_{\mathcal{V}} =
1$, by construction.

Turning our attention to the second ratio, let us compute the dividing surface
partition function for the coarsened system
first.  By definition,
\begin{equation*}
 Z^{\cg,\neq}_{\mathcal{V}}
 =
  \int_{\Gamma_{AB}^{\cg}}
    e^{-\beta\mathcal{V}^{\cg}_{s}(\mathbf{q}^{r},\beta)}
  dS^{r} .
\end{equation*}
Using the result for $\mathcal{V}^{\cg}_{s}$ from
\eqref{Eq:CGDividingSurfaceEnergy}, we have that
\begin{equation}\label{Eq:IntermediateCGDivZ}
 Z^{\cg,\neq}_{\mathcal{V}}
 =
 e^{-\beta\mathcal{V}(\mathbf{q}_{s})}
 \sqrt{\frac{(2\pi/\beta)^{dN^{c}}}{\text{det}\,\mathbf{C}}}
 \int_{\Gamma_{AB}^{\cg}}
   e^{-\frac{\beta}{2} \mathbf{u^{r}} \cdot \mathbf{D}^{\cg}\mathbf{u}^{r}}
 d\mathbf{u}^{r}.
\end{equation}
Now, $\mathbf{D}^{\cg}$ is a real, symmetric matrix.  
Let $\lambda^{\cg}$ denote the single
negative eigenvalue of $\mathbf{D}^{\cg}$, and let $\mathbf{v}^{\cg}$ denote its
corresponding normalized eigenvector.  Let $\lambda^{\cg}_{i}$ for $2 \leq i
\leq dN^{r}$ denote the remaining positive eigenvalues of the matrix with
$\mathbf{v}^{\cg}_{i}$ being their associated normalized eigenvectors which we
may choose so that they are orthonormal with respect to one another.  Now, for
any $\mathbf{q}^{r} \in \Gamma^{\cg}_{AB}$, the displacement $\mathbf{u}^{r} =
\mathbf{q}^{r} - \mathbf{q}^{r}_{s} \in \R^{dN^{r}}$ must be orthogonal to
$\mathbf{v}^{\cg}$ as the normal vector to the dividing surface is parallel to
this unstable eigenmode.  Thus, the displacement $\mathbf{u}^{r}$ may be written
as $\mathbf{u}^{r} = \sum_{i=2}^{dN^{r}}\alpha_{i}\mathbf{v}^{\cg}_{i}$ for some
real constants $\alpha_{i}$. Hence,
\begin{equation*}
 \mathbf{u}^{r} \cdot \mathbf{D}^{\cg} \mathbf{u}^{r}
 =
  \left(\sum_{i=2}^{dN^{r}}\alpha_{i}\mathbf{v}^{\cg}_{i}\right)\cdot
  \left(\sum_{i=2}^{dN^{r}}\alpha_{i}\mathbf{D}^{\cg}\mathbf{v}^{\cg}_{i}\right)
 =
 \sum_{i=2}^{dN^{r}}\alpha^{2}_{i}\lambda^{\cg}_{i}.
\end{equation*}
All of the eigenvalues in the sum are positive.  Therefore,
\begin{equation*}
\begin{split}
 \int_{\Gamma^{\cg}_{AB}}
   e^{-\frac{\beta}{2}\mathbf{u}^{r}\cdot\mathbf{D}^{\cg}\mathbf{u}^{r}}
 d\mathbf{u}^{r}
 &=
 \int_{\R^{dN^{r} - 1}}
   e^{-\frac{\beta}{2}\sum_{i=2}^{dN^{r}}\lambda_{i}\alpha_{i}^{2}}
 d\alpha_{2} \cdots d\alpha_{dN^{r}}
 \\ &=
 \left(\frac{2\pi}{\beta}\right)^{\frac{dN^{r} - 1}{2}}
 \frac{1}{\sqrt{\Pi_{i=2}^{dN^{r}}\lambda^{\cg}_{i}}}
 =
 \left(\frac{2\pi}{\beta}\right)^{\frac{dN^{r} - 1}{2}}
 \sqrt{\frac{|\lambda^{\cg}|}{|\text{det}\,\mathbf{D}^{\cg}|}}.
\end{split}
\end{equation*}
Substituting this result into  \eqref{Eq:IntermediateCGDivZ}, we see that
\begin{equation*}
 Z^{\cg,\neq}_{\mathcal{V}}
 =
 e^{-\beta\mathcal{V}(\mathbf{q}_{s})}
 \left(\frac{2\pi}{\beta}\right)^{\frac{dN - 1}{2}}
 \sqrt{\frac{|\lambda^{\cg}|}
      {\text{det}\,\mathbf{C}\,|\text{det}\,\mathbf{D}^{\cg}|}}.
\end{equation*}
A similar computation for $Z^{\neq}_{\mathcal{V}}$ yields
\begin{equation*}
 Z^{\neq}_{\mathcal{V}}
 =
 e^{-\beta\mathcal{V}(\mathbf{q}_{s})}
   \left(\frac{2\pi}{\beta}\right)^{\frac{dN-1}{2}}
   \sqrt{\frac{|\lambda^{\at}|}
              {|\text{det}\,\mathbf{D}^{\at}|}},
\end{equation*}
where $\lambda^{\at}$ is the single negative eigenvalue of
$\mathbf{D}^{\at}$.  With this result and the block-matrix identity from
\eqref{Eq:BlockDeterminantIdentity}, the desired ratio is
\begin{equation*}
 \frac{Z^{\cg,\neq}_{\mathcal{V}}}{Z^{\neq}_{\mathcal{V}}}
 = \sqrt{\frac{\lambda^{\cg}}{\lambda^{\at}}}.
\end{equation*}

Since $Z_{\mathcal{V}}/Z^{\text{cg}}_{\mathcal{V}} = 1$ and
since we proved in Theorem~\ref{Thm:WellBehaved} that
$\lambda^{\cg} /\lambda^{\at}\ge 1,$ the relative error for the TST rate approximation made by the coarsened system
can be shown to satisfy
\begin{equation}\label{Eq:RelativeErrorResult}
\frac{R^{\text{cg}}_{A \to B}-R^{TST}_{A \to B}}
       {R^{TST}_{A \to B}}
 = \frac{Z^{\text{cg},\neq}_{\mathcal{V}}}
                  {Z^{\neq}_{\mathcal{V}}}
             -1
 = \sqrt{\frac{\lambda^{\cg}}{\lambda^{\at}}} - 1\ge 0.
\end{equation}
Thus, the relative error in the TST rate computation is entirely dependent
upon the imaginary eigenfrequencies of the two dynamical matrices.  In addition,
we have that
\begin{equation*}
  R^{\cg}_{A \to B} \geq R^{TST}_{A \to B}.
\end{equation*}
To better understand this error, we will further investigate the
eigenvalue $\lambda^{\cg}$.
\begin{remark}
The relative error in the TST rate found above has no dependence
on temperature.  This is a consequence of the harmonic approximation of the
potential energy used in the beginning of the analysis.  If higher order terms
in the potential energy approximation are included, a temperature dependence in
the relative error will result.  This dependence on the thermodynamic
temperature $\beta$ will be $\mathcal{O}(\beta^{-2})$ so that this additional
error term goes to zero in the zero temperature limit.
\end{remark}

\section{Coarse-Grained Eigenvalue Analysis}

In the previous section, it was shown that the relative error in the TST rate
is entirely dependent upon the negative eigenvalues of the dynamical matrices
for the fully atomistic and coarse-grained systems at their respective
transition states.  To better understand this error, it is important to
understand how the negative eigenvalue for the fully atomistic system is
affected by the coarsening process.  This analysis will provide greater insight
into how the coarsened system relates to the original system as well as for
which situations the coarse-grained approximation of the TST rate will be most
accurate.  This insight will be useful in that it will be suggestive of
optimal approaches to coarse-graining a given problem.

In this section, we will again let $\mathbf{D}^{\at}$ and
$\mathbf{D}^{\cg}$ represent the dynamical matrices as defined previously for
the fully atomistic and coarse-grained systems.  We will let $\lambda^{\at}$
denote the single negative eigenvalue of $\mathbf{D}^{\at}$ while
$\mathbf{u}^{\at}$ will denote a normalized eigenvector corresponding to
$\lambda^{\at}$.  As before, we will also let $\lambda^{\cg}$ denote the sole
negative eigenvalue of $\mathbf{D}^{\cg}$, and we will let $\mathbf{v}^{\cg}$
denote a normalized eigenvector associated with this eigenvalue.  The sign of
$\mathbf{v}^{\cg}$ will be chosen so that $\mathbf{u}^{\at,r} \cdot
\mathbf{v}^{\cg} \geq 0$.  Here, $\mathbf{u}^{\at,r} \in \R^{dN^{r}}$ denotes a
vector consisting of only the repatom elements from the fully atomistic unstable
eigenmode.  Note that this element does not have the same dimension as
$\mathbf{u}^{\at}$.  Such a convention will be used throughout the remainder of
the paper when we wish to consider only the repatom or constrained portion of
a given variable.  To be clear, when using $c$ as a superscript, it
implies that the vector under consideration is an element of $\R^{dN^{c}}$.

To begin the analysis, let us determine the conditions necessary
for $\lambda^{\at} = \lambda^{\cg}$, which would imply no error in the
coarse-grained approximation of the TST rate.
\begin{theorem}\label{Thm:NecessaryConditions}
If $\lambda^{\at} = \lambda^{\cg}$, then
$\mathbf{u}^{\at,r}/\|\mathbf{u}^{\at,r}\| = \mathbf{v}^{\cg}$.  \em In
addition,  we have that
\begin{equation}\label{Eq:ConstrainedForceDifference}
 \mathbf{B}(\mathbf{u}^{\at,c}_{\textup{\text{min}}} - \mathbf{u}^{\at,c})
 =
 \mathbf{0},
\end{equation}
where $\mathbf{u}^{\at,c}_{\textup{\text{min}}} :=
-\mathbf{C}^{-1}\mathbf{B}^{\T}\mathbf{u}^{\at,r}$.
\end{theorem}
\begin{proof}
Suppose that $\lambda^{\at} = \lambda^{\cg}$ and that
$\mathbf{u}^{\at,c}_{\text{min}}$ is as defined in the statement of the
theorem.  From the definition of $\mathbf{D}^{\at}$ in
\eqref{Eq:AtDynamicalMatrix}, we see that $\mathbf{D}^{\at}\mathbf{u}^{\at} =
\lambda^{\at}\mathbf{u}^{\at}$ implies that
\begin{equation*}
\mathbf{R}\mathbf{u}^{\at,r} +
\mathbf{B}\mathbf{u}^{\at,c} = \lambda^{\at}\mathbf{u}^{\at,r}.
\end{equation*}
  Using this
fact, we have from the definition of $\mathbf{D}^{\cg}$ in \eqref{Eq:CGDynMatrix} that
\begin{equation}\label{Eq:CGIdealEigenvector}
 \mathbf{D}^{\cg}\mathbf{u}^{\at,r}
 =
 \mathbf{R}\mathbf{u}^{\at,r} + \mathbf{B}\mathbf{u}^{\at,c}_{\text{min}}
 =
 \lambda^{\at}\mathbf{u}^{\at,r}
   + \mathbf{B}(\mathbf{u}^{\at,c}_{\text{min}} - \mathbf{u}^{\at,c}).
\end{equation}
It will be shown later in the proof that $\mathbf{u}^{\at,r} \cdot
\mathbf{B}(\mathbf{u}^{\at,c}_{\text{min}} -
\mathbf{u}^{\at,c}) \leq 0$ and that $\|\mathbf{u}^{\at,r}\| \neq 0$.  For now,
let us assume that these two statements are true.  As $\lambda^{\cg}$
is the absolute minimum of the quadratic form
$\mathbf{v}\cdot\mathbf{D}^{\cg}\mathbf{v}$ subject to the constraint
$\|\mathbf{v}\| = 1$, we may use  \eqref{Eq:CGIdealEigenvector} and our
recent assumptions to show that
\begin{equation}\label{Eq:DirectEigenvalueInequality}
 \lambda^{\cg}
 \leq
 \frac{\mathbf{u}^{\at,r} \cdot \mathbf{D}^{\cg}\mathbf{u}^{\at,r}}
      {\|\mathbf{u}^{\at,r}\|^{2}}
 = \lambda^{\at} +
 \frac{\mathbf{u}^{\at,r}\cdot\mathbf{B}(\mathbf{u}^{\at,c}_{\text{min}} -
                                         \mathbf{u}^{\at,c})}
      {\|\mathbf{u}^{\at,r}\|^{2}}
 \leq
 \lambda^{\at}.
\end{equation}
Since $\lambda^{\at} = \lambda^{\cg}$, all of the inequalities in the above
result are equalities.  The first inequality that is now an equality implies
that $\mathbf{u}^{\at,r}/\|\mathbf{u}^{\at,r}\|$ is a normalized eigenvector of
$\mathbf{D}^{\cg}$ associated with $\lambda^{\cg}$.
The choice $\mathbf{v}^{\cg} := \mathbf{u}^{\at,r}/\|\mathbf{u}^{\at,r}\|$ satisfies $\mathbf{u}^{\at,r} \cdot
\mathbf{v}^{\cg} \geq 0$.
Making the appropriate substitutions into
\eqref{Eq:CGIdealEigenvector}, we now have that
\begin{equation*}
 \mathbf{D}^{\cg}\mathbf{v}^{\cg} = \lambda^{\cg}\mathbf{v}^{\cg} +
   \frac{\mathbf{B}(\mathbf{u}^{\at,c}_{\text{min}} - \mathbf{u}^{\at,c})}
        {\|\mathbf{u}^{\at,r}\|}
\end{equation*}
from which we immediately see that $\mathbf{B}(\mathbf{u}^{\at,c}_{\text{min}} -
\mathbf{u}^{\at,c}) = \mathbf{0}$.

To finish the proof, we will now prove the two claims made earlier in the proof.
Observe that
\begin{equation}\label{ba}
 \mathbf{u}^{\at,r}\cdot\mathbf{B}\mathbf{u}^{\at,c}_{\text{min}}
 =
 -\mathbf{u}^{\at,r} \cdot
    \mathbf{B}\mathbf{C}^{-1}\mathbf{B}^{\T}\mathbf{u}^{\at,r}
 =
 -\mathbf{B}^{\T}\mathbf{u}^{\at,r} \cdot
   \mathbf{C}^{-1}\mathbf{B}^{\T}\mathbf{u}^{\at,r}.
\end{equation}
Now, we may use the definition of
$\mathbf{D}^{\at}$ and the equation $\mathbf{D}^{\at}\mathbf{u}^{\at} =
\lambda^{\at}\mathbf{u}^{\at}$ to show that
$\mathbf{B}^{\T}\mathbf{u}^{\at,r}  + \mathbf{C}\mathbf{u}^{\at,c} =
\lambda^{\at}\mathbf{u}^{\at,c}$.  If $\mathbf{u}^{\at,r}$ were a
zero vector, we would arrive at the contradictory conclusion
that $\mathbf{C}$ has a negative eigenvalue.  As this is not
the case, $\mathbf{u}^{\at,r}$ is not a zero vector so that
$\|\mathbf{u}^{\at,r}\| \neq 0$.  Rearranging the equation arising from the
definition of $\mathbf{D}^{\at}$, we have that
\begin{equation}\label{aa}
\mathbf{B}^{\T}\mathbf{u}^{\at,r}  + (\mathbf{C} -
\lambda^{\at}\mathbf{I})\mathbf{u}^{\at,c} =
\mathbf{0},
\end{equation}
where $\mathbf{I}$ is an identity matrix of the appropriate
dimensions.  The eigenvalues of $(\mathbf{C} - \lambda^{\at}\mathbf{I})$ may
easily be shown to be the eigenvalues of $\mathbf{C}$ plus $|\lambda^{\at}|$
with the same corresponding eigenvectors as we are only adding a scalar multiple
of the identity matrix to $\mathbf{C}$. Since $\mathbf{C}$ is a
positive-definite matrix and all of its eigenvalues are positive, adding the
positive number $|\lambda^{\at}|$ to the eigenvalues of $\mathbf{C}$ does not
change their sign. Hence, the eigenvalues of $(\mathbf{C} -
\lambda^{\at}\mathbf{I})$ are all positive, so this matrix is invertible. Thus,
after some further manipulation of \eqref{aa}, we can show that
\begin{equation}\label{bb}
 \mathbf{B}\mathbf{u}^{\at,c}
 =
 -\mathbf{B}(\mathbf{C} -
             \lambda^{\at}\mathbf{I})^{-1}\mathbf{B}^{\T}\mathbf{u}^{\at,r}.
\end{equation}
Combining \eqref{ba} with \eqref{bb}, we can obtain
\begin{equation*}
 \mathbf{u}^{\at,r} \cdot \mathbf{B}(\mathbf{u}^{\at,c}_{\text{min}} -
   \mathbf{u}^{\at,c})
 =
 -\mathbf{B}^{\T}\mathbf{u}^{\at,r} \cdot
 (\mathbf{C}^{-1} - (\mathbf{C} - \lambda^{\at}\mathbf{I})^{-1})\mathbf{B}^{\T}
   \mathbf{u}^{\at,r}.
\end{equation*}
From the earlier comment regarding the eigenvalues and eigenvectors of
$(\mathbf{C} - \lambda^{\at}\mathbf{I})$, we see that the matrix
$(\mathbf{C}^{-1} - (\mathbf{C} - \lambda^{\at}\mathbf{I})^{-1})$ is in fact
positive definite. Note that this statement follows from the assumption that
$\lambda^{\at}$ is strictly negative, but allowing $\lambda^{\at} = 0$ does not
change the following result as the aforementioned matrix would still be positive
semi-definite.  In either case,
\begin{equation*}
 \mathbf{u}^{\at,r} \cdot \mathbf{B}(\mathbf{u}^{\at,c}_{\text{min}} -
   \mathbf{u}^{\at,c})
 \leq 0.
\end{equation*}
The claims are proven, so the proof is complete.  
\end{proof}
The converse of the above theorem is true as well; that is, if
\eqref{Eq:ConstrainedForceDifference} holds and $\mathbf{u}^{\at,r}$ is an
eigenvector of $\mathbf{D}^{\cg}$, then $\lambda^{\at} = \lambda^{\cg}$.
A simple proof of this statement is to substitute
the assumption of the form of $\mathbf{u}^{\at,r}$ and
\eqref{Eq:ConstrainedForceDifference} into
\eqref{Eq:CGIdealEigenvector}.  In fact,
\eqref{Eq:ConstrainedForceDifference} alone is sufficient to prove that the
eigenvalues must be identical.  The theorem, then, shows that in order for no
error to be made in the coarse-graining approximation of the TST rate that the
constrained atoms in the unstable eigenmode must interact with the repatom
region as if the constrained atoms were in their relaxed configuration.  Note
that the result does not necessarily imply that $\mathbf{u}^{\at,c}_{\text{min}}
= \mathbf{u}^{\at,c}$ as the kernel of $\mathbf{B}$ may be non-trivial.  As
$\mathbf{B}$ is affected by the range of interactions among the constituents in
a system, it is not difficult to construct an example where the kernel of
this matrix would be non-trivial.

One particular case of interest for an exact coarse-grained approximation of
the TST rate occurs when the only non-zero components in the unstable eigenmode
are those that lie in the chosen repatom region, or equivalently, when
$\mathbf{u}^{\at,c} = \mathbf{0}$. In such a case, the behavior of interest is
extremely localized, so we should not be surprised that we do not lose any
information by coarse-graining the components of the system that have no
involvement in the transition.  This implies that the coarse-graining method
works well when the unstable eigenmode is localized; that is, we should expect
the coarse-graining scheme to be accurate when the repatom region contains the
atoms which have the largest contribution to the norm of the unstable eigenmode
$\mathbf{u}^{\at}$. Accurately approximating the TST rate in such an instance
was the primary motivation for the development of this method. It is also
interesting to note that  \eqref{Eq:DirectEigenvalueInequality} provides
another proof that $\lambda^{\cg} \leq \lambda^{\at}$.

The presence of $\mathbf{u}^{\at,c}_{\text{min}}$ in the above condition for
no error to be made in the coarse-graining approximation can be explained
through the following theorem:
\begin{theorem}\label{Thm:Embedding}
Let $\mathbf{v} \in \R^{dN^{r}}$ and let
\begin{equation*}
 \mathbf{v}_{\textup{\text{min}}} :=
   \begin{bmatrix}
    \mathbf{v} \\
    -\mathbf{C}^{-1}\mathbf{B}^{\T}\mathbf{v}
   \end{bmatrix}.
\end{equation*}
Then,
\begin{equation*}
 \mathbf{D}^{\at}\mathbf{v}_{\textup{\text{min}}} =
   \begin{bmatrix}
    \mathbf{D}^{\cg}\mathbf{v} \\
    \mathbf{0}
   \end{bmatrix}.
\end{equation*}
Thus, $\mathbf{v}\cdot\mathbf{D}^{\cg}\mathbf{v} =
\mathbf{v}_{\textup{\text{min}}}\cdot\mathbf{D}^{\at}\mathbf{v}_{\textup{\text{
min}}}$. In particular, $\lambda^{\cg} =
\mathbf{v}^{\cg}\cdot\mathbf{D}^{\cg}\mathbf{v}^{\cg}
= \mathbf{v}^{\cg}_{\textup{\text{min}}} \cdot
\mathbf{D}^{\at}\mathbf{v}^{\cg}_{\textup{\text{min}}}$, where
\begin{equation*}
 \mathbf{v}^{\cg}_{\textup{\text{min}}} :=
   \begin{bmatrix}
    \mathbf{v}^{\cg} \\
    -\mathbf{C}^{-1}\mathbf{B}^{\T}\mathbf{v}^{\cg}
   \end{bmatrix}.
\end{equation*}
\end{theorem}
\begin{proof}
Let the vectors be as defined in the problem statement.  Then,
\begin{equation*}
  \mathbf{D}^{\at}\mathbf{v}_{\text{min}}
  =
  \begin{bmatrix}
   \mathbf{R} & \mathbf{B} \\
   \mathbf{B}^{\T} & \mathbf{C}
  \end{bmatrix}
  \begin{bmatrix}
   \mathbf{v} \\
   -\mathbf{C}^{-1}\mathbf{B}^{\T}\mathbf{v}
  \end{bmatrix}
  =
  \begin{bmatrix}
   (\mathbf{R} - \mathbf{B}\mathbf{C}^{-1}\mathbf{B}^{\T})\mathbf{v} \\
   \mathbf{0}
  \end{bmatrix}
  =
  \begin{bmatrix}
   \mathbf{D}^{\cg}\mathbf{v} \\
   \mathbf{0}
  \end{bmatrix}.
\end{equation*}
The remaining results are immediate consequences of the above equation. 
\end{proof}

Recall from  \eqref{Eq:RelaxedConstrainedAtoms} that the matrix
$-\mathbf{C}^{-1}\mathbf{B}^{\T}$ yields the relaxed constrained atom
configuration for a given repatom configuration; that is, the matrix finds the
displacement vector for the constrained atoms that minimizes the total energy of
the configuration. Thus, using the notation in the above theorem, it has been
shown that $\mathbf{v} \mapsto \mathbf{v}_{\text{min}}$ is a linear,
one-to-one mapping of the coarse-grained phase space into a subspace of the
fully atomistic phase space that preserves energy differences between
configurations as well as forces.  We see then that the coarse-graining method
removes the constrained degrees of freedom while still taking into account their
presence by always considering the constrained atoms to be in their
energy-minimizing state. The vector $\mathbf{B}(\mathbf{u}^{\at,c}_{\text{min}}
- \mathbf{u}^{\at,c})$ tells how much the constrained system's actual behavior
through the transition state deviates from this energy-minimizing assumption.
This difference could also be interpreted as a measure of how well the coarsened
system captures the boundary conditions for the repatom region due to the
long-range effect of the constrained atoms.  Note that this result was a
consequence of the form of the dynamical matrix as was discussed at the end of
the derivation of the dividing surface for the coarse-grained system.
Also, recall that the range of the function $\mathbf{v} \mapsto
\mathbf{v}_{\text{min}}$ which came about from the use of this method is
extremely similar to the subspace of the fully atomistic phase space
utilized in the static version of the hot-QC methods.

We can use elements of the first theorem in this section to derive an
expression for the difference between $\lambda^{\at}$ and $\lambda^{\cg}$:
\begin{theorem}\label{Thm:ErrorBound}
In general,
\begin{equation}\label{est}
 \lambda^{\cg} - \lambda^{\at}
 =
 \frac{\mathbf{v}^{\cg} \cdot
       \mathbf{B}(\mathbf{u}^{\at,c}_{\textup{min}} - \mathbf{u}^{\at,c})}
      {\mathbf{v}^{\cg} \cdot \mathbf{u}^{\at,r}}.
\end{equation}
%
%
%
\end{theorem}
\begin{proof}
Recall from \eqref{Eq:CGIdealEigenvector} that
\begin{equation*}
 \mathbf{D}^{\cg}\mathbf{u}^{\at,r}
 =
 \lambda^{\at}\mathbf{u}^{\at,r}
   + \mathbf{B}(\mathbf{u}^{\at,c}_{\text{min}} - \mathbf{u}^{\at,c}).
\end{equation*}
Taking the dot product of both sides of the equation with $\mathbf{v}^{\cg}$,
we have that
\begin{equation*}
 \lambda^{\cg}\mathbf{v}^{\cg} \cdot \mathbf{u}^{\at,r}
 =
 \mathbf{v}^{\cg} \cdot \mathbf{D}^{\cg}\mathbf{u}^{\at,r}
 =
 \lambda^{\at}\mathbf{v}^{\cg} \cdot \mathbf{u}^{\at,r}
   +  \mathbf{v}^{\cg} \cdot
      \mathbf{B}(\mathbf{u}^{\at,c}_{\text{min}} - \mathbf{u}^{\at,c}).
\end{equation*}
The result \eqref{est} follows if $\mathbf{v}^{\cg} \cdot
\mathbf{u}^{\at,r} \neq 0.$  To show this, note that if $\mathbf{v}^{\cg}
\cdot \mathbf{u}^{\at,r} = 0$, then $\mathbf{u}^{\at,r} \cdot
\mathbf{D}^{\cg}\mathbf{u}^{\at,r} \geq 0$ as $\mathbf{v}^{\cg}$ is the only
eigenvector of the symmetric matrix $\mathbf{D}^{\cg}$ with a negative
eigenvalue. However, by \eqref{Eq:QuadraticPortionofAtomisticEnergy} we have
that
\begin{equation*}
 \mathbf{u}^{\at,r} \cdot \mathbf{D}^{\cg} \mathbf{u}^{\at,r}
 \leq
 \mathbf{u}^{\at} \cdot \mathbf{D}^{\at} \mathbf{u}^{\at}
 =
 \lambda^{\at} < 0.
\end{equation*}
Thus, we have a contradiction, so the claim is proven.  
\end{proof}

The above error \eqref{est} can be broken down into three components: namely,
the error in the approximation of $\lambda^{\at}$ by $\lambda^{\cg}$ is affected
by the rotation of the dividing surface in the coarse-grained phase space, how
much of the essential components of the transition are captured in the repatom
region, and how well the unstable eigenmode is approximated by the
coarsened system.  We have already mentioned that the
$\mathbf{B}(\mathbf{u}^{\at,c}_{\text{min}} - \mathbf{u}^{\at,c})$
quantity is a measure of this third error.  The dot product of this
quantity with $\mathbf{v}^{\cg}$ picks out the portion of the resulting force
that impacts the transition.  The remaining errors are reflected in
the $\mathbf{u}^{\at,r} \cdot \mathbf{v}^{\cg}$ term.  Recall that
$\mathbf{u}^{\at,r}$ is the repatom component of the normal vector to the
dividing surface in the fully atomistic phase space while $\mathbf{v}^{\cg}$ is
the normal vector to the coarse-grained dividing surface.  The geometric
definition of the dot product states that
\begin{equation*}
 \mathbf{u}^{\at,r} \cdot \mathbf{v}^{\cg}
 =
 \|\mathbf{u}^{\at,r}\| \cos(\theta),
\end{equation*}
where $\theta$ is the angle between $\mathbf{u}^{\at,r}$ and
$\mathbf{v}^{\cg}$.  The angle $\theta$ represents how much the dividing
surface is rotated as it is projected into the coarse-grained phase space.  As
the mismatch between the direction of the vectors $\mathbf{u}^{\at,r}$ and
$\mathbf{v}^{\cg}$ increases, the error in the coarse-grained approximation of
the TST rate will increase.  Physically, this increase is caused by the
coarse-grained dividing surface passing through a lower-energy region of the
phase space as a result of the rotation.  The remaining portion of the error
term, $\|\mathbf{u}^{\at,r}\|^{-1}$, characterizes how much of the essential
components of the transition are captured within the repatom region as has
been mentioned earlier.  Note that the magnitude of an individual component
of the unstable eigenmode determines the relative importance of that
component to the transition. With these three errors in mind, we could also
write the error \eqref{est} found in the theorem as
\begin{equation*}
  \lambda^{\at} - \lambda^{\cg}
  =
  \frac{\mathbf{v}^{\cg} \cdot
        \mathbf{B}(\mathbf{u}^{\at,c}_{\text{min}} - \mathbf{u}^{\at,c})}
       {\|\mathbf{u}^{\at,r}\|\cos(\theta)}.
\end{equation*}
%

 Above, it was shown that by
including all of the atoms that contribute significantly to the localized
transition, the coarse-grained approximation of the TST rate would be accurate.
The error derived in Theorem \ref{Thm:ErrorBound} suggests that the error
in the coarse-grained approximation can be further reduced by choosing the
repatom region in such a way so as to minimize the error due to
$\mathbf{v}^{\cg} \cdot \mathbf{B}(\mathbf{u}^{\at,c}_{\text{min}} -
\mathbf{u}^{\at,c})$.  Additional refinement of the repatom region to minimize
this error contribution would be similar to what is already done in the
quasicontinuum methods when choosing a mesh for the continuum region and will be
demonstrated in the next section on numerical results.  More interestingly, this
error formulation seems to imply the possibility that the coarse-graining
problem may be approached with the primary goal of minimizing
$\mathbf{v}^{\cg} \cdot \mathbf{B}(\mathbf{u}^{\at,c}_{\text{min}}
- \mathbf{u}^{\at,c})$. In such an approach, it may not be as necessary
to fully capture the localized region of interest in the repatom region provided
this boundary condition norm can be made sufficiently small.  This
perspective on the error suggests that there might be
problems for which this strategy is well-suited. The tradeoff between
these two terms will be further discussed below.

The relative error between the two transition rates depends on the ratio
between $\lambda^{\cg}$ and $\lambda^{\at}$.  We can, of course, use the above
theorem to write this ratio as
\begin{equation*}
 \frac{\lambda^{\cg}}{\lambda^{\at}}
 =
 1 + \frac{1}{|\lambda^{\at}|}
   \frac{\mathbf{v}^{\cg} \cdot
         \mathbf{B}(\mathbf{u}^{\at,c}_{\text{min}} - \mathbf{u}^{\at,c})}
        {\mathbf{u}^{\at,r} \cdot \mathbf{v}^{\cg}}.
\end{equation*}
\section{Numerical Results}

\begin{figure}[h]
\begin{center}
\subfigure[$s =
1.02$\label{fig:FractureUnstableEigenmode:Ex1}]{\includegraphics[width= 6.75
cm]{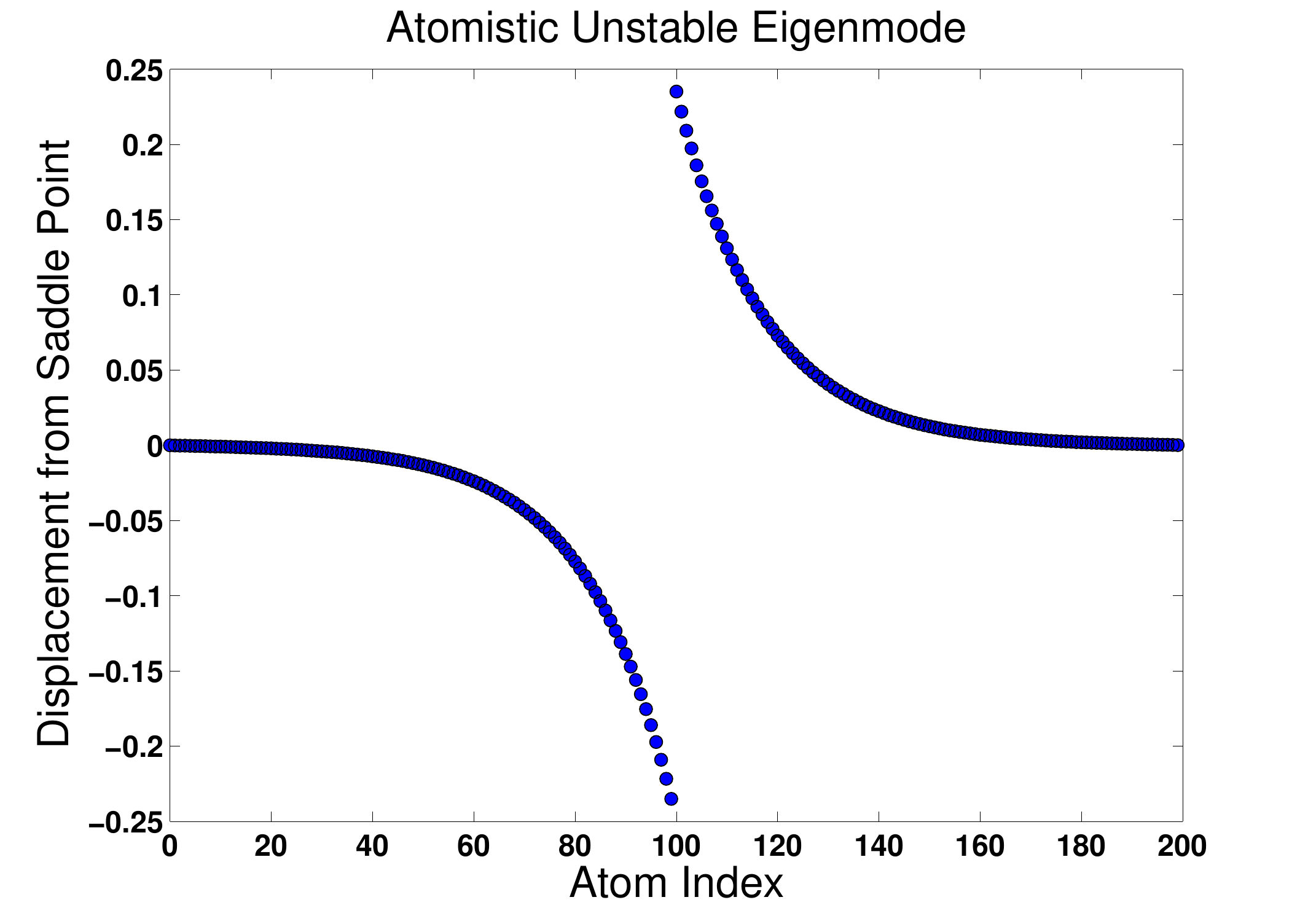}}
~
\subfigure[$s =
1.035$\label{fig:FractureUnstableEigenmode:Ex2}]{\includegraphics[width= 6.75
cm]{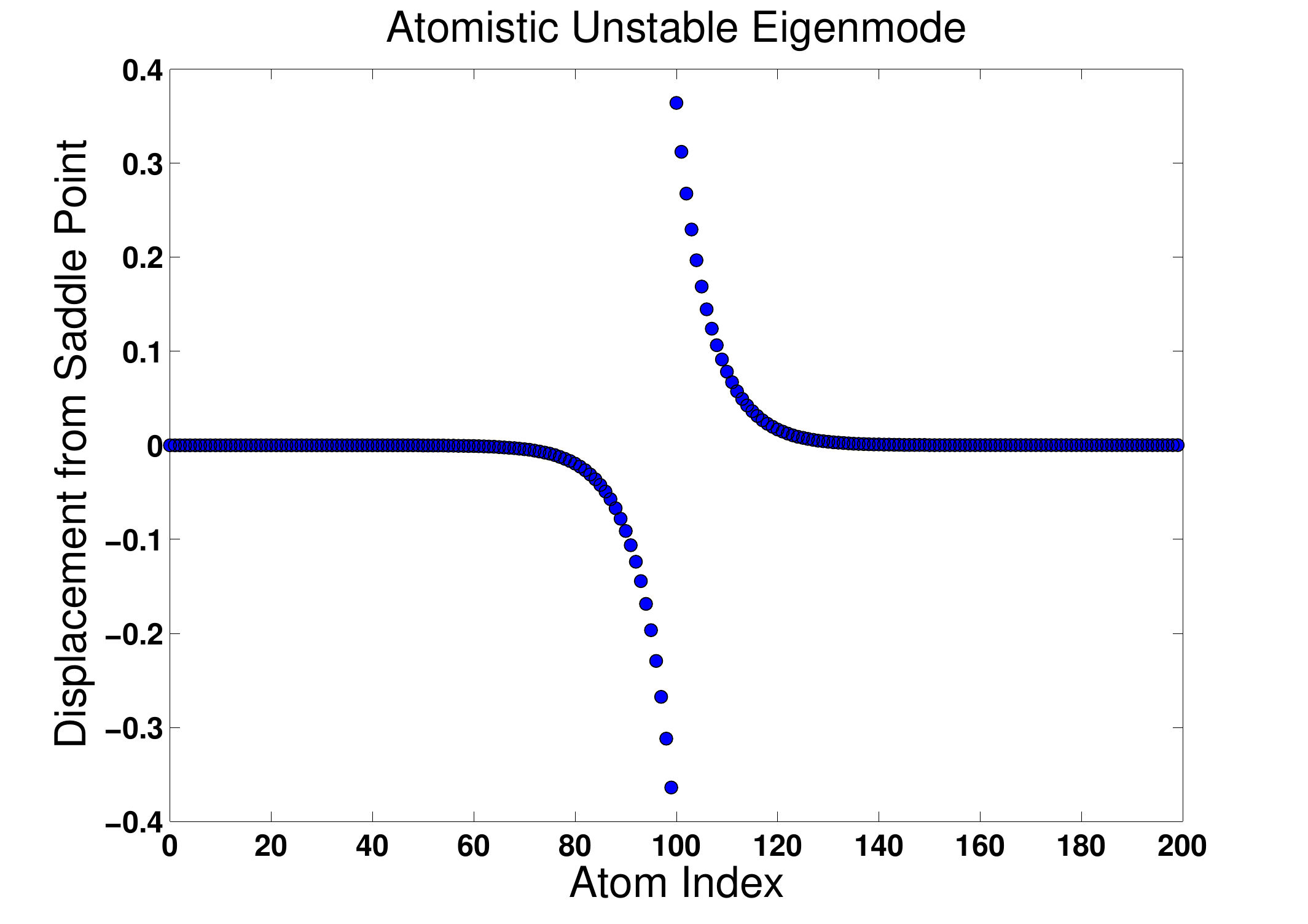}}

\end{center}
\caption{The unstable eigenmode for the fully
atomistic 1D chain for two different tensile
strains determined by the scalar $s$.}\label{fig:FractureUnstableEigenmode}
\end{figure}

\begin{figure}[t]
    \centering
    \subfigure{\includegraphics[scale=0.5]{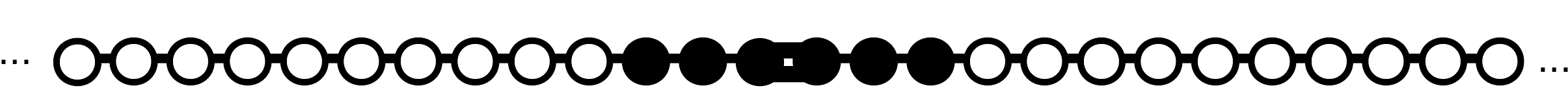}}
    \subfigure{\includegraphics[scale=0.5]{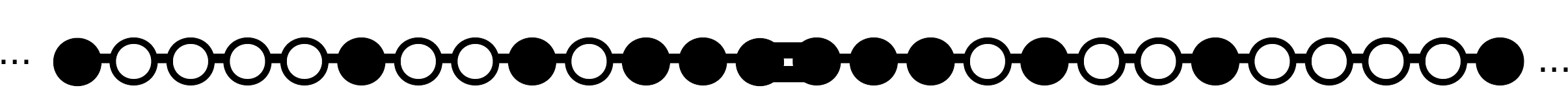}}
    \caption{Illustration of the localized and delocalized
           coarse-graining schemes for a core region containing 6
           atoms. Circles represent atoms in the system: filled
           circles represent repatoms while the empty circles represent
           constrained atoms.  The core region is the collection of the 6
           contiguous repatoms in the center of the chain.  The weakened bond
           is represented by the set of two lines connecting the two central
           atoms in the figure.}
    \label{fig:MeshSchemes}
\end{figure}

\begin{figure}[t]
\begin{center}
\subfigure[$s =
1.02$\label{fig:FractureCGUnstableEigenmode:Ex2}]{\includegraphics[width= 6.75
cm]{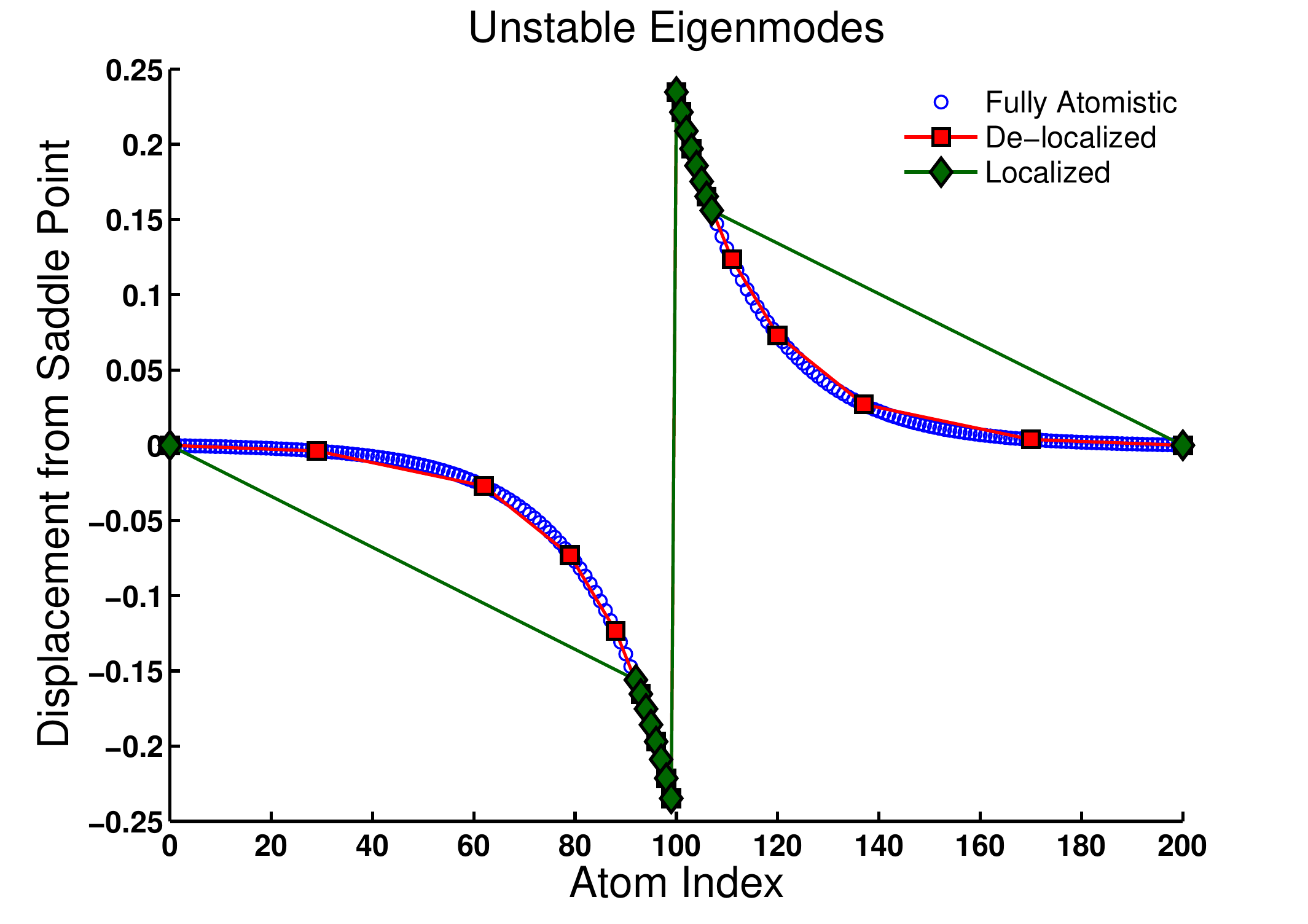}}
~
\subfigure[$s =
1.035$\label{fig:FractureCGUnstableEigenmode:Ex1}]{\includegraphics[width= 6.75
cm]{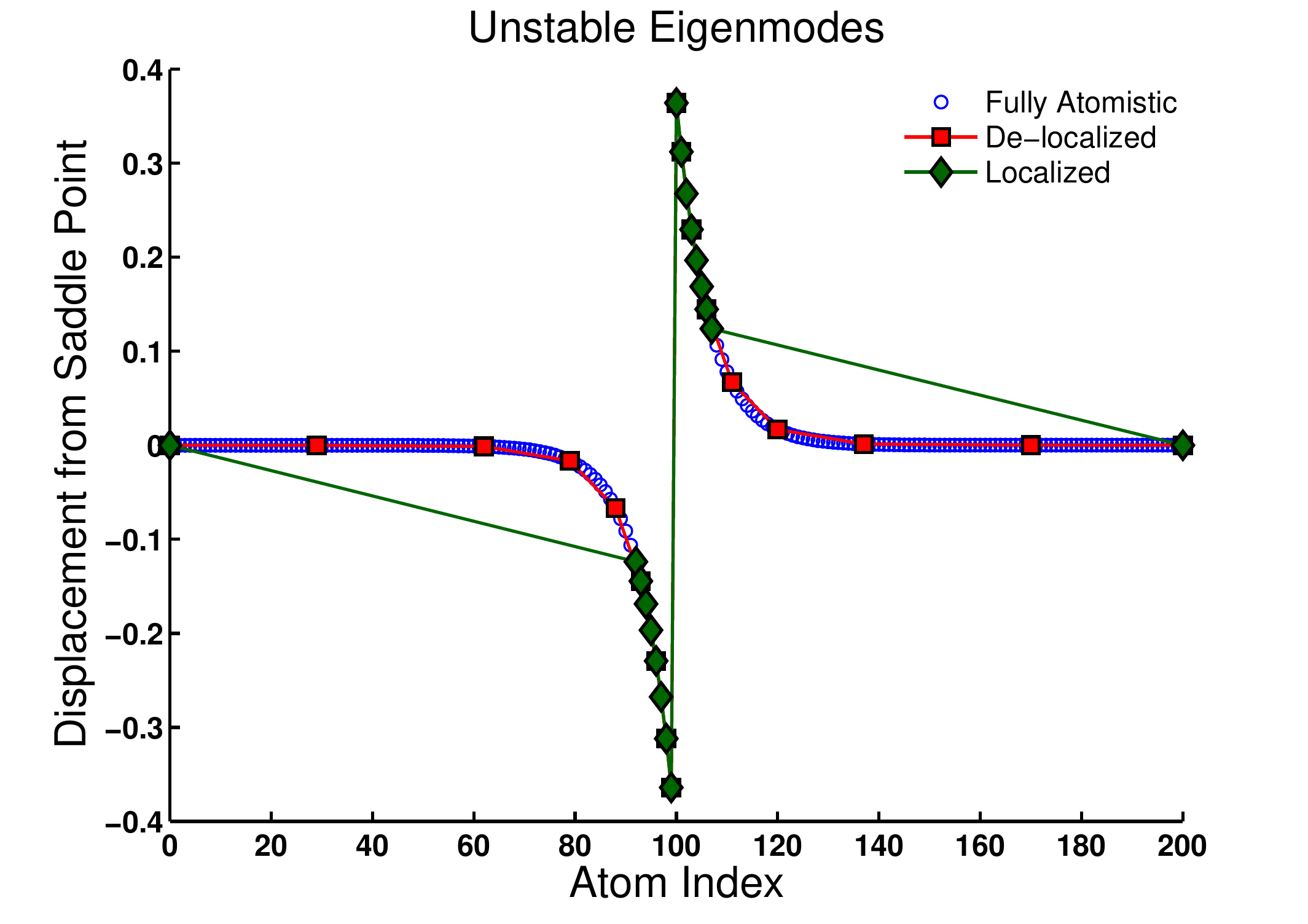}}
\end{center}
\caption{Comparison of the fully atomistic unstable eigenmode and the
coarse-grained unstable eigenmodes computed for the localized and delocalized
coarse-graining schemes for the two different tensile strains.  The repatoms for each of
the coarse-graining schemes are indicated by the markers.  The localized and delocalized
coarse-graining schemes contain the same number of degrees of freedom in both graphs.
}\label{fig:FractureCGUnstableEigenmode}
\end{figure}

\begin{figure}[t]
\begin{center}
\subfigure[$s =
1.02$\label{fig:FractureEigenvalueComparison:Ex2}]{\includegraphics[width= 6.75
cm]{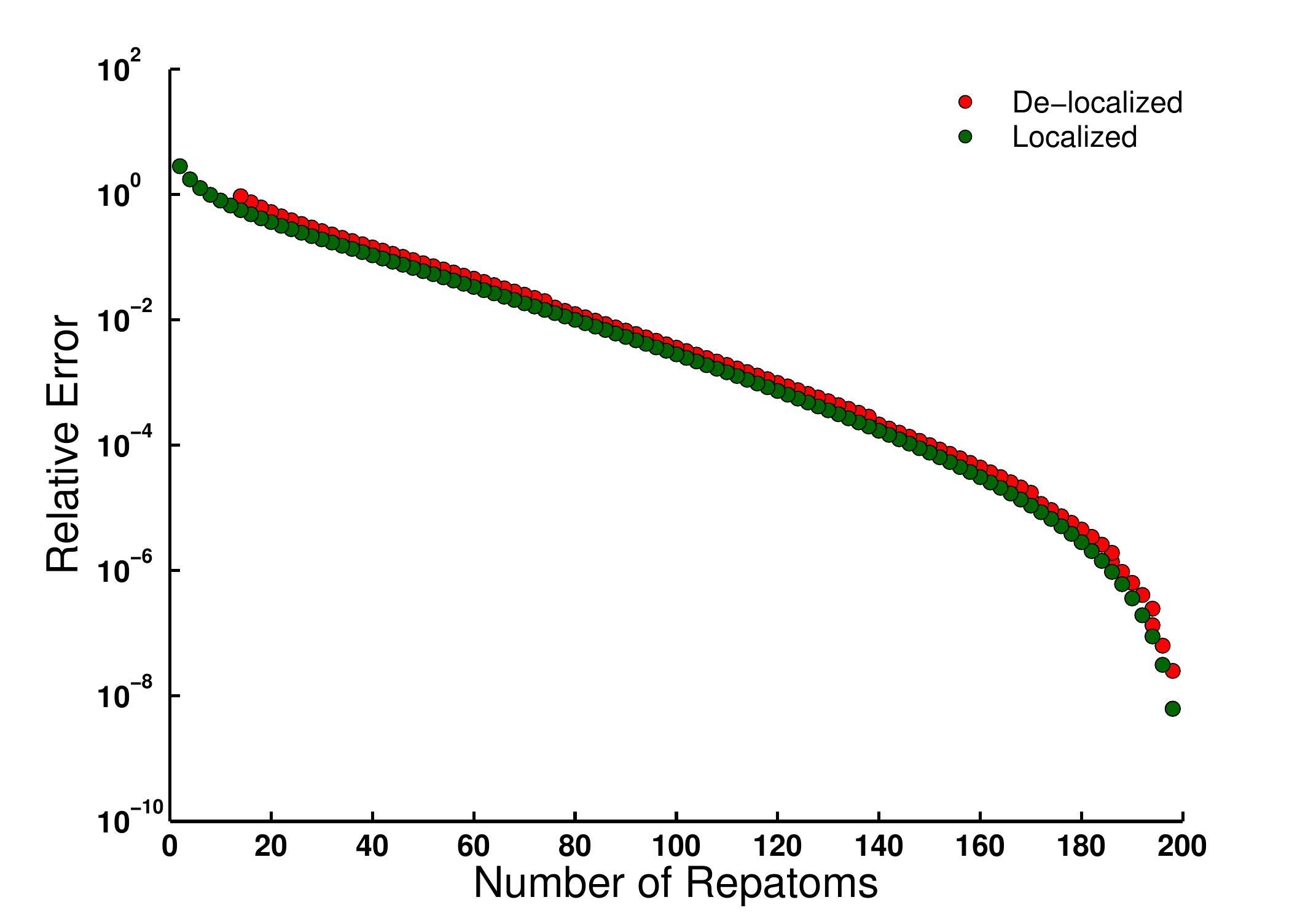}}
~
\subfigure[$s =
1.035$\label{fig:FractureEigenvalueComparison:Ex1}]{\includegraphics[width= 6.75
cm]{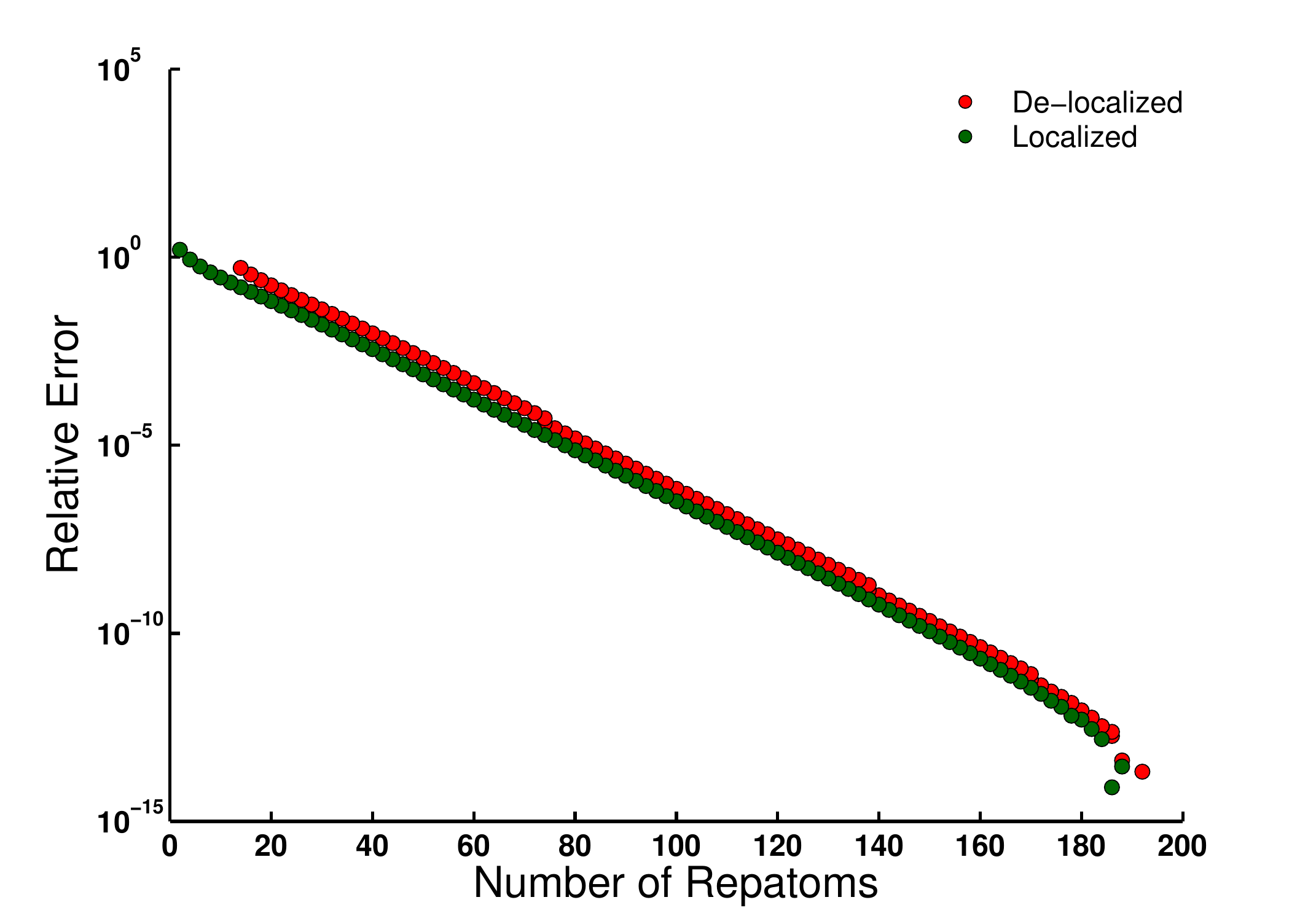}}
\end{center}
\caption{Relative error in the TST rate approximation, computed from
$\sqrt{\lambda^{\text{cg}}/\lambda^{\text{at}}} - 1$, for the localized and
delocalized coarse-graining schemes for two different tensile
strains.}\label{fig:FractureEigenvalueComparison}
\end{figure}

\begin{figure}[t]
\begin{center}
\subfigure[$s =
1.02$\label{fig:FractureDotProductComparison:Ex2}]{\includegraphics[width= 6.75
cm]{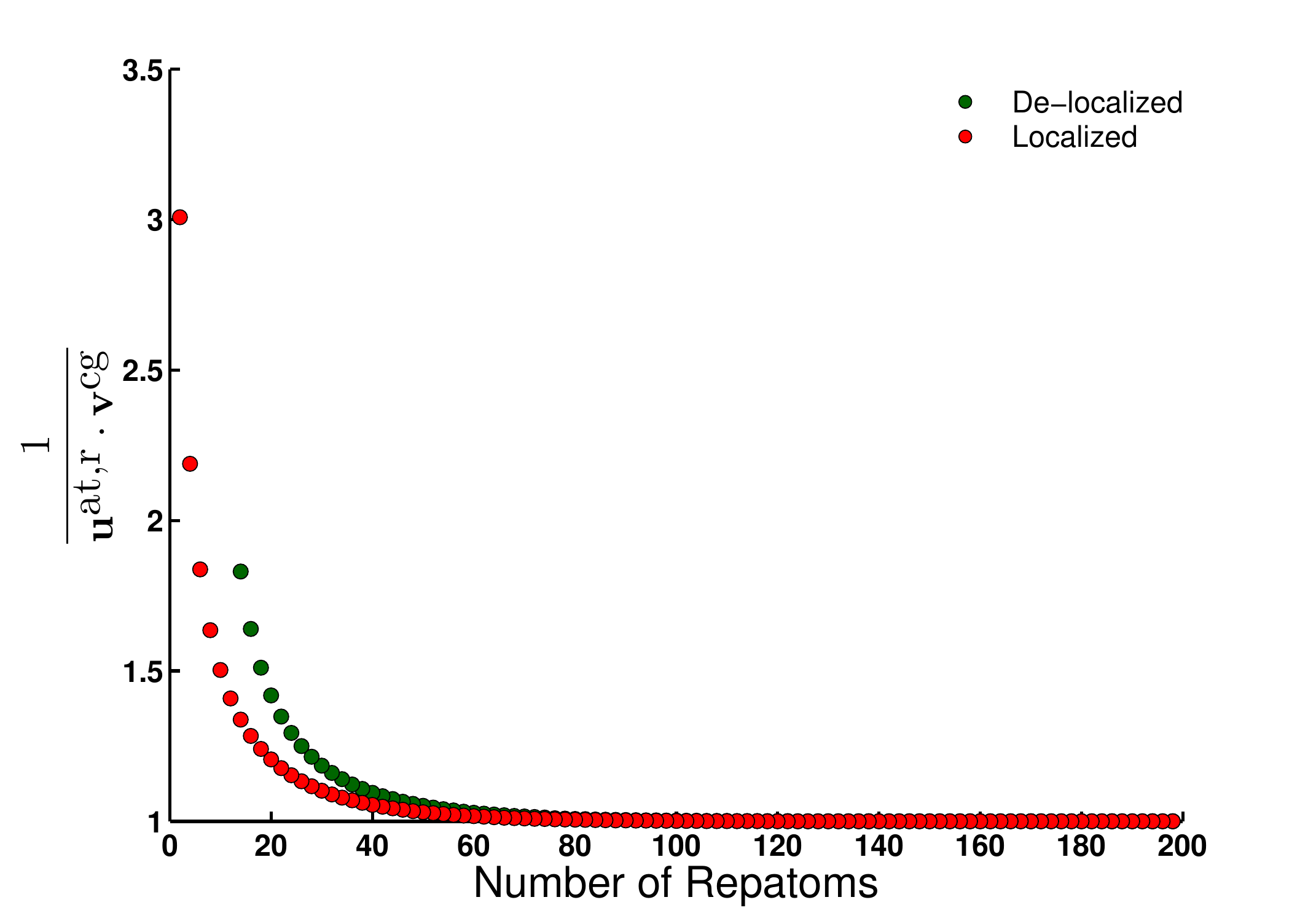}}
~
\subfigure[$s =
1.035$\label{fig:FractureDotProductComparison:Ex1}]{\includegraphics[width= 6.75
cm]{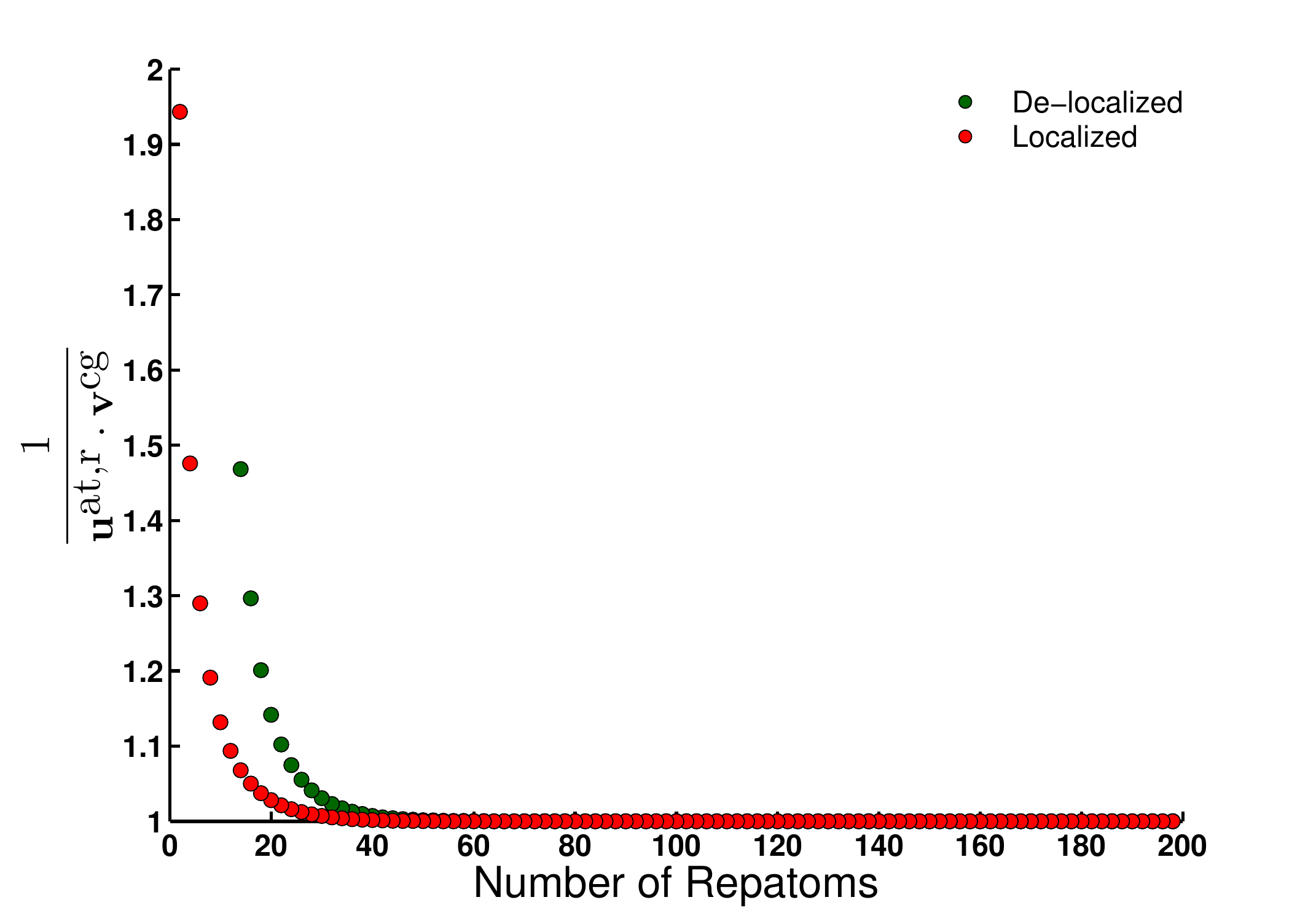}}
\end{center}
\caption{The denominator of the error term from Theorem
\ref{Thm:ErrorBound} for the localized and delocalized approaches to coarse
graining for two different tensile strains that measure the rotation of the
dividing surface and the relevant portions of the transition captured in the
repatom region. Ideally, this term should be equal to
1.}\label{fig:FractureDotProductComparison}
\end{figure}

\begin{figure}[t]
\begin{center}
\subfigure[$s =
1.02$\label{fig:FractureLongRangeComparison:Ex2}]{\includegraphics[width= 6.75
cm]{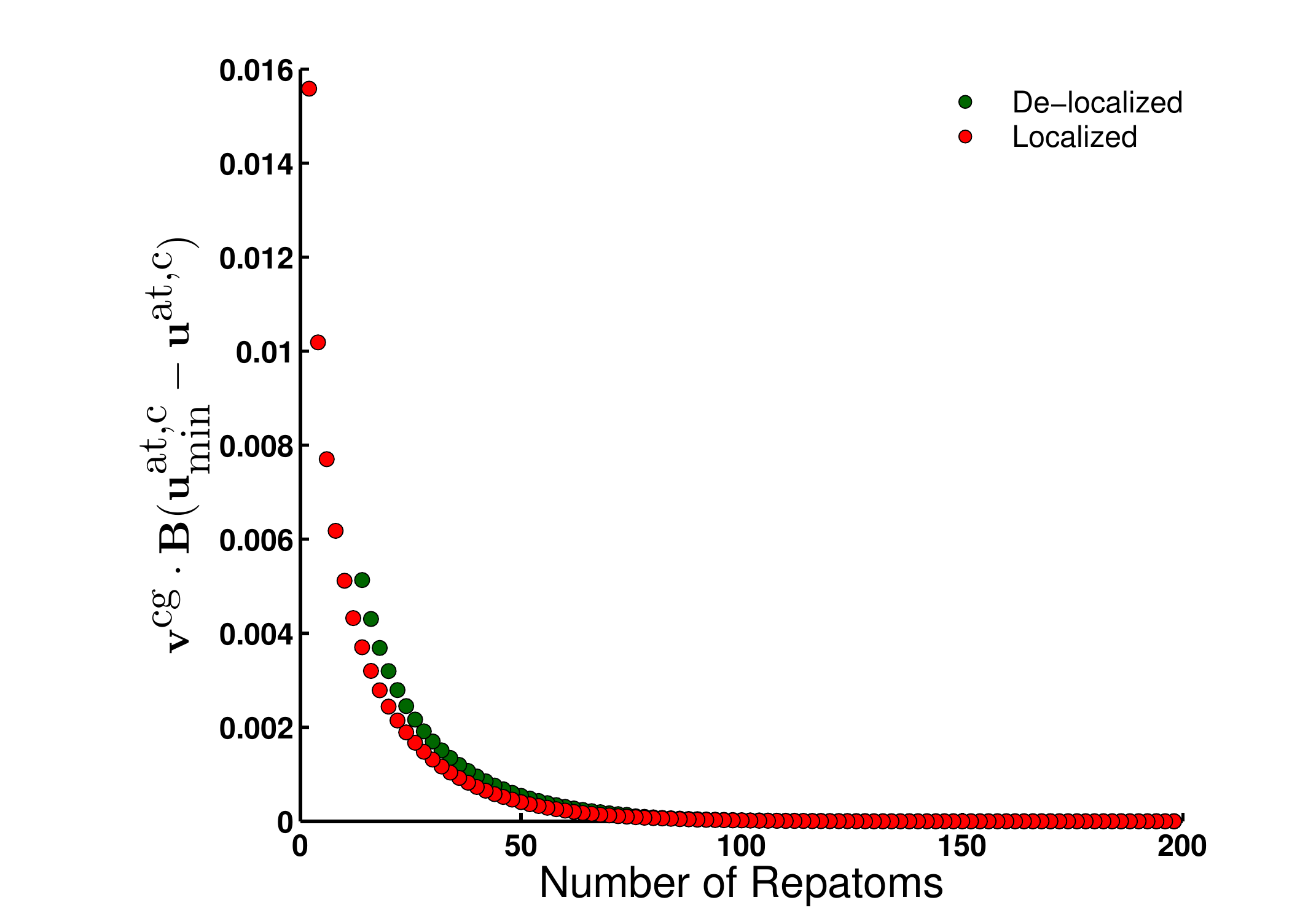}}
~
\subfigure[$s =
1.035$\label{fig:FractureLongRangeComparison:Ex1}]{\includegraphics[width= 6.75
cm]{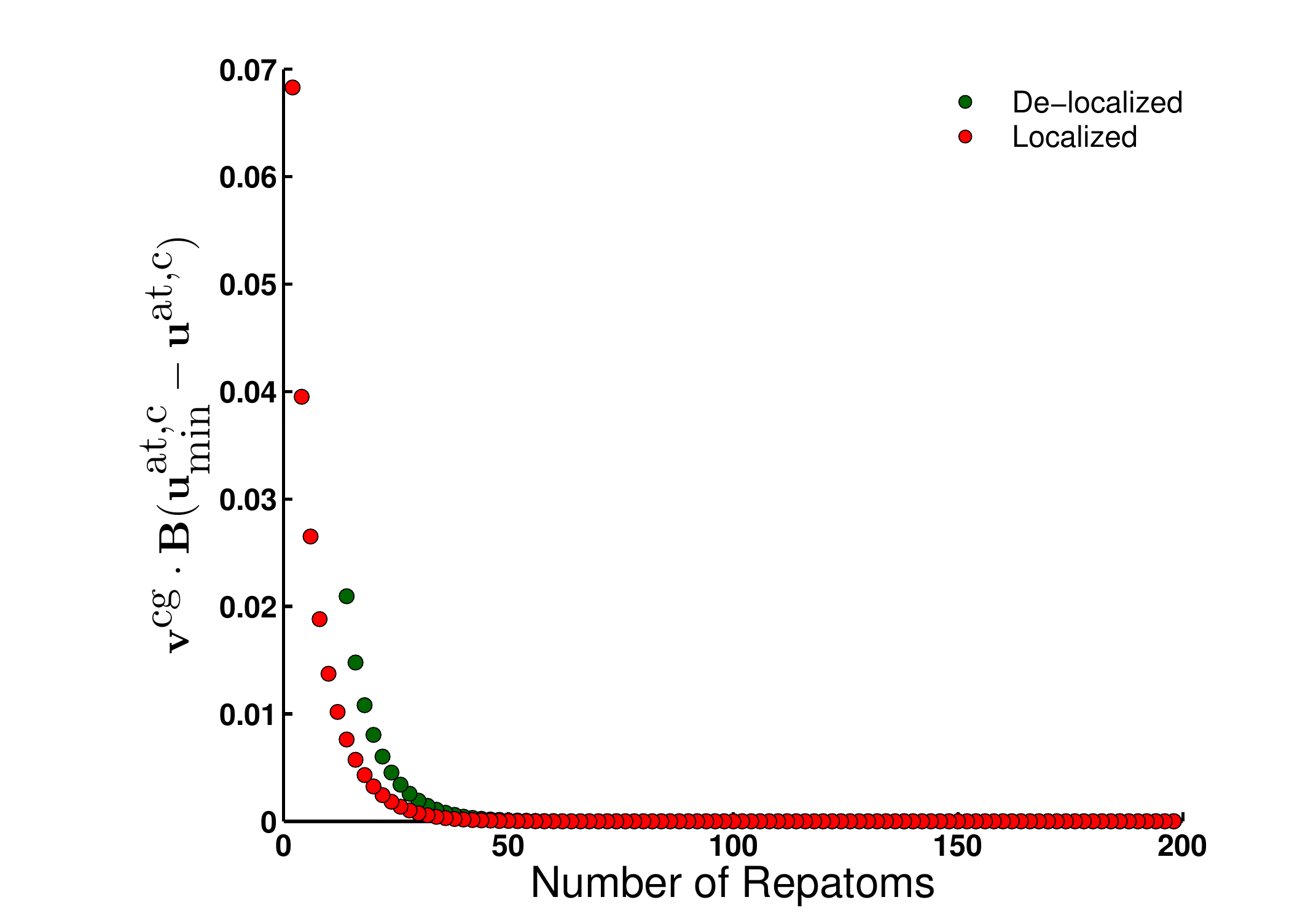}}
\end{center}
\caption{Long-range elastic contribution to the error from Theorem
\ref{Thm:ErrorBound} for the localized and delocalized approaches to
coarse-graining for two different tensile
strains.}\label{fig:FractureLongRangeComparison}
\end{figure}

In this section, we seek to verify through numerical experiments that the
coarse-graining method described in this paper is indeed able to accurately
reproduce the fully atomistic TST rate and to compare the error between
qualitatively different approaches to coarsening a system with a
localized region of interest.  For the first coarse-graining scheme, the repatom region will be
chosen with the intent of maximizing the magnitude of the projection
of the unstable eigenmode onto the repatom space.  As the region of interest is
localized, this approach implies that the repatoms should be concentrated in this
region as well.  In the second coarse-graining scheme, the repatom region will consist of a core
region in the localized region of interest with additional repatoms placed
throughout the remainder of the system so as to better capture the
long-range effect that the constrained region has on the
repatoms.  We will refer to these coarse-graining schemes as the localized repatom mesh and
delocalized repatom mesh schemes, respectively.

The system that will be considered in the numerical experiments is a 1-D chain
of atoms with fixed endpoints.  Only nearest-neighbor interactions are
considered.  All of the atoms in the chain interact through the
same potential except for those atoms which form the central bond,
whose interaction potential is made weaker. We are interested
in the rate at which this weakened bond breaks, causing the fracture
of the chain. The process is expected to primarily involve the atoms nearest to
the central bond, resulting in a localized transition.  The localized nature of
this transition will be demonstrated in the numerical experiments.
Note that this system is closely related to that studied in \cite{hyperqc}.

The 1-D chain consists of 202 atoms.
The energy contribution due to the central bond in this chain will take the form
of a Lennard-Jones potential:
\begin{equation*}
 \mathcal{V}_{c}(r)
 =
 4\varepsilon\left(\left(\frac{\sigma}{r}\right)^{12} -
                   \left(\frac{\sigma}{r}\right)^{6}\right),
\end{equation*}
where $\varepsilon = 1$, $\sigma = \frac{1}{2^{1/6}}$, and $r$ is the length
of the bond.  The remaining bonds in the chain are treated as harmonic
springs; i.e., the potential energy contribution of a single bond is
of the form:
\begin{equation*}
 \mathcal{V}(r)
 =
 \frac{1}{2}(r - 1)^{2}.
\end{equation*}
Note that the equilibrium bond
length for this potential is 1.  Letting $\mathbf{q} = (q_{i})_{i=0}^{201}$
denote the position of the atoms in the chain ($q_{i}$ indicating
the position of the $(i + 1)$-th atom), the total energy of the chain is
\begin{equation*}
 \mathcal{V}_{\text{chain}}(\mathbf{q})
 =
 \mathcal{V}_{c}(q_{101} - q_{100})
 +
 \sum_{i=1, i \neq 101}^{201}\mathcal{V}(q_{i} - q_{i-1})
 .
\end{equation*}
We set the boundary conditions for the endpoints of the chain so that $q_{0}
= 0$ and $q_{201} = 201s$, where $s = 1.02$ or $s = 1.035$.
The purpose of the scalar $s$ is to impose a tensile strain on the
chain and make fracture energetically favorable.  Changing the tensile strain
affects the degree of locality of the transition region.

The accuracy of the approximation of the TST rate will be discussed in terms of
 \eqref{Eq:RelativeErrorResult}, so we need only compute the negative
eigenvalues of the dynamical matrices discussed in the previous
sections.  Note that we use the same notation for the dynamical
matrices, eigenmodes, eigenvalues, etc., in this section as we have before.  To
determine the negative eigenvalue belonging to $\mathbf{D}^{\at}$, we first
compute the saddle point or transition state, $\mathbf{q}_{s}$, of the system
just described.  In our numerical simulations, we found the saddle point by
slowly stretching the weakened bond at the center of the chain while
simultaneously relaxing the remaining atoms until all of the forces in the chain were found to be
zero within a given tolerance. We then numerically compute
$\mathbf{D}^{\at}$ and diagonalize it to determine both $\lambda^{\at}$
and $\mathbf{u}^{\at}$.  The unstable eigenmode $\mathbf{u}^{\at}$ is shown in
Figure \ref{fig:FractureUnstableEigenmode} for the two different tensile
strains. It is clear from the picture that this transition is fairly well
localized in both cases, as the atoms closest to the central bond
are the largest contributors to the norm of $\mathbf{u}^{\at}$.
Physically, this unstable eigenmode simply shows that as the unbroken chain
crosses over to the broken state, the central atoms in the two regions move away
from one another as indicated by the displacements in the eigenvector.  When
the strain is greater, the fracture of the chain is more localized.

For this problem, it is possible to analytically determine the saddle point.
Before we begin with the derivation of the saddle point, recall that the length
of the chain under consideration is $201s$. For clarity, the length of the chain
will be denoted by $L$ in this section of the analysis.  Now, due to the
symmetry of the system about the central bond, we expect the saddle point to
display a similar symmetry. That is, we expect that $q_{i} - q_{0} = q_{201} -
q_{201 - i}$ for $0 \leq i \leq 100$.  As a consequence of this result, we can
write the central bond length solely in terms of $q_{100}$.  Explicitly,
$q_{101} - q_{100} = L - 2q_{100}$.  Because of the strictly convex nature of
the spring potential and the symmetry in the atomic positions, it is also
possible to show that every bond length governed by the spring potential, there
are 200 such bonds, is equal to the same value.  The bond lengths partition the
length of the chain minus the central bond length, so we may compute the bond
length to be $\frac{L - (L - 2q_{100})}{200} = \frac{q_{100}}{100}$.  With this
result and the alternate formula for the central bond length, we have reduced
the problem of computing the saddle point down to simply computing $q_{100}$.
To compute this value, let us consider the balance of the forces on the 100th
atom in the chain.  This atom is part of the central bond and interacts via the
Lennard-Jones potential with the 101st atom, but it interacts by the spring
potential with the atom with index 99.  At the saddle point, these forces
should cancel.  Thus, the force balance equation is
\begin{equation*}
 (q_{100} - q_{99}) + 4\varepsilon\left(12\frac{\sigma^{12}}{(q_{101} -
q_{100})^{13}} - 6\frac{\sigma^{6}}{(q_{101} - q_{100})^{7}}\right) = 0.
\end{equation*}
Substituting our results for the spring bond length and the central bond
length, this equation becomes
\begin{equation*}
 \frac{q_{100}}{100} + 4\varepsilon\left(12\frac{\sigma^{12}}{(L -
2q_{100})^{13}} -
6\frac{\sigma^{6}}{(L - 2q_{100})^{7}}\right) = 0
\end{equation*}
This non-linear equation can be turned into a 14-degree polynomial.  The roots
of this resulting polynomial that lie in $(0, \frac{1}{2}L)$ give the
possible values of $q_{100}$ in the saddle point.  As this is the only position
necessary to determine the location of every atom in the saddle point, the roots
of the polynomial give the transition state of the problem.

Note that with the bond lengths between the atoms in the chain known, we can
derive an analytical expression for the unstable eigenvector
$\mathbf{u}^{\at}$ in terms of the negative eigenvalue $\lambda^{\at}$.  To see
this, let $\mathbf{u}^{\at}_{100}$ be the displacement of the 100th atom, which
is the
leftmost atom that interacts via the Lennard-Jones potential, and recall that
the endpoints of the chain are fixed, so we can take the displacement
$\mathbf{u}^{\at}_{0} = 0$.  Applying $\mathbf{D}^{\at}$ to the unstable
eigenmode
$\mathbf{u}^{\at}$ and assuming that $\mathbf{u}^{\at}_{100}$ is known, we see
that
the displacements in the eigenmode for $\mathbf{u}^{\at}_{i}$ for $0 < i < 100$
may be
determined from a second-order difference equation with two boundary conditions
given by $\mathbf{u}^{\at}_{0}$ and $\mathbf{u}^{\at}_{100}$. Specifically, we
have that
\begin{equation*}
 -\mathbf{u}^{\at}_{i-1} + 2\mathbf{u}^{\at}_{i} - \mathbf{u}^{\at}_{i+1} =
\lambda^{\at}\mathbf{u}^{\at}_{i} \quad \text{for} \; 0 < i < 100
\end{equation*}
with $\mathbf{u}^{\at}_{100}$ taken to be some constant to be determined from
normalization and $\mathbf{u}^{\at}_{0} = 0$. Solving this difference equation
yields the solution
\begin{equation*}
 \mathbf{u}^{\at}_{i} = \alpha r_{+}^{i} + \beta r_{-}^{i} \quad \text{for} \;
 0 \leq i \leq 100,
\end{equation*}
where
\begin{equation*}
 \alpha
 =
 \frac{-\mathbf{u}^{\at}_{100}}{r_{-}^{100} - r_{+}^{100}},
 \quad
 \beta
 =
 \frac{\mathbf{u}^{\at}_{100}}{r_{-}^{100} - r_{+}^{100}},
\end{equation*}
and
\begin{equation*}
 r_{\pm}
 =
 1 - \frac{\lambda^{\at}}{2} \pm
 \sqrt{\frac{\lambda^{\at}}{2}\left(\frac{\lambda^{\at}}{2} - 2\right)}.
\end{equation*}
Note that $r_{+} > 1$ while $r_{-} < 1$.  Due to the symmetry of the problem, we
can get a similar equation for the displacements on the right-hand side of the
chain.  The displacements of the two central atoms that interact via the
Lennard-Jones potential are determined from the normalization of the eigenmode
and the fact that $\mathbf{u}^{\at}_{100} = -\mathbf{u}^{\at}_{101}$, which is
due to the symmetry intrinsic to the problem.




As stated previously, the localized repatom coarse-graining scheme intends to
maximize $\|\mathbf{u}^{\at,r}\|$.  This was accomplished in the numerical
experiments by constraining a continuous line of atoms at one end of the chain
and the mirror image of this grouping at the chain's other end, leaving a
contiguous repatom region at the center of the chain.  The total number of
repatoms was then varied. For the delocalized coarse-graining scheme, the selection of the
repatom region began with the inclusion of a contiguous region of repatoms
centered around the central bond. Additional repatoms were placed in the
periphery with the spacing between them increasing geometrically moving away
from the core region. Specifically, the spacing was doubled after starting with
a single constrained atom between the core region and the first repatom in the
periphery.  Following this, there would be two, then four, eight, etc.,
constrained atoms between each pair of repatoms until the end of the chain was
reached.  In symbols, we may write the set of indices for the $N$ repatoms in
the localized coarse-graining scheme in the following way:
\begin{equation*}
 \text{Localized Indices}(N) =
   \left\{100 - \ell : 0 \leq \ell \leq \frac{1}{2}N - 1 \right\}
   \bigcup \left\{101 + \ell : 0 \leq \ell \leq \frac{1}{2}N - 1 \right\}.
\end{equation*}
Of course, here, $N$ must be an even integer with $2 \leq N \leq 200$.  In
symbols, we may write the set of indices for the repatoms in the delocalized
coarse-graining scheme with $N$ repatoms in the core in the following way:
\begin{align*}
 \text{De}&\text{-localized Indices}(N) = \text{Localized Indices}(N)
   \\ \bigcup &
   \left\{100 - \left(\frac{1}{2}N - 1\right) - 2^{\ell} - (\ell - 1) :
     \ell \geq 1 \; \text{and}
   \; 100 - \left(\frac{1}{2}N - 1\right) - 2^{\ell} - (\ell - 1) > 0
   \right\}
   \\ \bigcup &
   \left\{101 + \left(\frac{1}{2}N - 1\right) + 2^{\ell} + (\ell - 1) :
     \ell \geq 1 \; \text{and}
   \; 101 + \left(\frac{1}{2}N - 1\right) + 2^{\ell} + (\ell - 1) < 201
     \right\}.
\end{align*}
The repatom configuration generated by this method is symmetric about
the central bond for the delocalized coarse-graining scheme. The number of repatoms in the core
region was then also varied. An illustration of these two coarse-graining
schemes is provided in Figure \ref{fig:MeshSchemes}.

Once a repatom set was defined for the experiment, $\mathbf{D}^{\at}$
and  \eqref{Eq:CGDynMatrix} were used to directly compute
$\mathbf{D}^{\cg}$. A diagonalization of this matrix then yielded
$\lambda^{\cg}$ and $\mathbf{v}^{\cg}$.  A comparison of the unstable
eigenmodes for the two coarse-graining schemes and the fully atomistic
unstable eigenmode for a given resolution are shown in Figure
\ref{fig:FractureCGUnstableEigenmode}.  The markers in the graph denote which
atoms were included in the repatom region for the experiments and provide
another illustration of the difference between the repatom regions used in the
localized and delocalized repatom mesh methods.

The relative error of the HTST
approximation, or $\sqrt{\frac{\lambda^{\cg}}{\lambda^{\at}}} - 1$, is shown in
Figure \ref{fig:FractureEigenvalueComparison} for varying numbers of
repatoms and for the two tensile strains.  The results show that the relative rate error
decreases extremely rapidly, i.e., roughly exponentially in this case, with
increasing numbers of repatoms. Achieving relative rate errors of less than 1\%
requires only about 40 to 50 degrees of freedom for both cases for $s=1.035$.  Further, the
localized coarse-graining scheme is seen to outperform the delocalized coarse-graining scheme for all
meshes we investigated although the difference is smaller for the more
delocalized transition.

To understand the cause of the difference in the accuracy of the two
methods, we look to Theorem \ref{Thm:ErrorBound} for the principle components
of the error.  In Figure \ref{fig:FractureDotProductComparison}, we see a
combination of the error due to not fully resolving some of the more essential
atoms in the transition and of the rotation of the dividing surface in the
coarse-grained phase space, reflected in the term ${\mathbf{u}^{\at,r} \cdot
\mathbf{v}^{\cg}}$, while the the error due to the long-range elastic
contributions, given by $\mathbf{v}^{\cg} \cdot
{\mathbf{B}(\mathbf{u}^{\at,c}_{\text{min}} -  \mathbf{u}^{\at,c})}$,
is shown in Figure \ref{fig:FractureLongRangeComparison}.  In both cases, the
localized coarse-graining scheme is seen to be preferable.  That the localized coarse-graining scheme had a
lower error contribution from the ${\mathbf{u}^{\at,r} \cdot
\mathbf{v}^{\cg}}$ term was expected
given that the aim of this coarse-graining scheme is the maximization of
$\|\mathbf{u}^{\at,r}\|$ by concentrating the repatoms in the region
where the components of $\mathbf{u}^{\at}$ are the largest in terms of absolute
value.  In contrast, the delocalized coarse-graining scheme distributes some repatoms away from
the fracture point, where the components of the unstable eigenmode are smaller
in norm, hence leading to suboptimal performance in this case. The better
performance of the localized method over the delocalized method in the case of
the long-range error is much more surprising and deserves extra attention as the
delocalized method was meant to reduce this error specifically.

To that end, let us consider the $\mathbf{v}^{\cg} \cdot
{\mathbf{B}(\mathbf{u}^{\at,c}_{\text{min}} -  \mathbf{u}^{\at,c})}$ term in
more detail by first considering the relaxed constrained configuration, or
$\mathbf{u}^{\at,c}_{\text{min}}$, for the two coarse-graining schemes.  The relaxed constrained
configuration is especially easy to determine for the present potential in both
cases as the potential is strictly convex in the periphery. For a line of
constrained atoms between two repatoms in this 1-D system, the relaxed
configuration is simply given by a linear interpolation of the displacement
between the two repatoms.  Therefore, we can compute
$\mathbf{u}^{\at,c}_{\text{min}}$ through a simple linear interpolation
between the nodes in the two coarse-graining schemes in the periphery.  This allows for an easy
comparison of $\mathbf{u}^{\at,c}_{\text{min}}$ and $\mathbf{u}^{\at,c}$.
The linear interpolation between the repatom components of the
coarse-grained unstable eigenmode is reported in Figure
\ref{fig:FractureCGUnstableEigenmode}.  It is quite evident from this viewpoint
that the approximation of the constrained region by the localized method is
worse than the approximation due to the delocalized repatom coarse-graining scheme.
The additional nodes in the periphery for the delocalized coarse-graining scheme help
better capture the long-range behavior of the system.  This difference is,
however, mitigated by the fact that we are here only considering
nearest-neighbor interactions. Therefore, the only difference in the constrained
region approximation that truly matters is the difference for the constrained
atoms that directly interact with the repatom region.  This is reflected in the
error derived in Theorem \ref{Thm:ErrorBound} through the kernel of
$\mathbf{B}$. For systems with longer-range interactions, it is conceivable that
this long-range effect error may become considerably more important.  With all
this in mind, the delocalized method should still perform better than the
localized method in terms of the $\|\mathbf{B}(\mathbf{u}^{\at,c}_{\text{min}}
-  \mathbf{u}^{\at,c})\|$ norm.  This is, in fact, the case.  The localized
method however performs better in the long-range error due to the contribution of the
dot product with $\mathbf{v}^{\text{cg}}$. As mentioned above, in the localized case, the
atoms with the largest values of $\mathbf{v}^{\cg}$ do not contribute
to the error as they are not coupled to the constrained atoms through
$\mathbf{B}$. They are thus effectively shielded by the core region
and only a small number of rep-atoms eventually contribute to the error.
In contrast, in the delocalized case, a larger number of repatoms
with significant values of  $\mathbf{v}^{\cg}$ contribute, tilting the
balance in favor of the localized coarse-graining scheme.

Another key point to keep in mind when considering the long-range error is the
mesh used in the periphery.  The mesh used in the example above is not the
optimal mesh for this system and was chosen instead as a realistic coarse-graining scheme to
use without knowing exactly the unstable eigenmode.  For this problem, many of
the degrees of freedom in the coarse-graining scheme do not contribute much to reducing the
long-range error, and this negatively affects the delocalized coarse-graining scheme in a
degree of freedom comparison against the localized coarse-graining scheme.  A more optimum
choice of repatoms in the periphery for the delocalized coarse-graining scheme would make the
comparison more favorable.  Theorem \ref{Thm:ErrorBound} could be used as a
starting point for a derivation of an optimal mesh as is done in
\cite{acta.atc} for quasicontinuum methods, but the dot products in
the error
present challenges to the derivation.  Using the Cauchy-Schwarz inequality
could help to alleviate this issue; however, the resulting bound is not a
good approximation to the actual result and the resulting mesh is suboptimal.
We can still consider better, if not optimal, meshes though to see if the
delocalized coarse-graining scheme will better handle the long-range error contribution.
For the simple delocalized coarse-graining scheme with a single repatom in the periphery
on either side of the chain with only a single constrained atom between these
peripheral atoms and the core, the delocalized method does indeed become
superior compared to the localized mesh in the long-range error for small core
regions. However, the localized coarse-graining scheme still ultimately had a lower relative
error even in this case.  The localized coarse-graining scheme was found to always have a lower
relative error than the delocalized coarse-graining scheme in all of the meshes examined in
the numerical experiments. Note that in actual implementations, the
delocalized coarse-graining scheme might possess other practical advantages. For
example, in 1D with nearest-neighbor interactions, $\mathbf{C}$ would
exhibit a block structure that could potentially be exploited.


Overall, for the original system considered here, the localized method is
superior in all aspects of the error. Consequently, choosing repatoms so as to
increase $\|\mathbf{u}^{\at,r}\|$ is the optimal strategy.  In higher
dimensions, this may change as the boundary region between the localized repatom
and periphery becomes more significant.

\section{Conclusions and Future Considerations}

In this paper, we have demonstrated that the CGMD approach to
atomistic coarse-graining can produce an accurate approximation to the TST rate
of the fully atomistic system.  Over the course of the analysis, we verified
that the coarsened system is well-behaved in that no spurious behaviors are
introduced through the coarsening process, and we described the projection of
the dividing surface into the coarse-grained phase space.  The error analysis
was extended to show under which circumstances the coarse-grained approximation
of the TST rate would be most accurate and highlighted the significant
contributions to the error in this estimation.  Our numerical results
demonstrated the accuracy of two different approaches to the coarse-graining
approximation in the context of a 1-D chain undergoing fracture.
The success of these approaches demonstrated that the
number of degrees of freedom taken into account to approximate the TST rate
could be significantly reduced while still maintaining a highly accurate
approximation.   While the localized method proved to be superior in the
numerical experiments performed here, the accuracy of the approximations made
by the two methods were comparable.  This is important to note because the
implementation of the delocalized method may be more efficient in certain
situations.

It is interesting to note that the analysis and error calculations for
this method are independent of the basis that is chosen for the problem.  While
it is certainly natural to consider a basis consisting solely of
individual atoms, it may be possible in certain situations to choose a basis for
the problem that further decreases the relevant number of degrees of
freedom such as the continuous, piecewise linear basis functions used in
\cite{PhysRevB.72.144104} and \cite{EBTadmor:2013,hyperqc}. Theoretically, it is possible to initially choose
an ideal basis consisting of the eigenvectors for the dynamical matrix at the
saddle point as described earlier in the paper.  In such a situation, only a
single basis element would contribute in any way to the TST rate allowing for a
coarsening of the system down to a single element while still maintaining a
perfect approximation of the TST rate.  While computing such an ideal basis is
usually not practical, especially if the point of the simulation is to
discover the appropriate escape transition \cite{PUSAV2009}, there may be
alternative choices for a basis that are relatively easy to work with, apply to
certain classes of problems, and still make the behavior of interest
increasingly localized.  For future consideration, it would also be interesting
to investigate the effect that longer-range interactions have on the
constrained region's contribution to the overall error, as discussed in the
numerical results section.

%


\begin{thebibliography}{10}

\bibitem{BlancLeBrisLegoll2005}
X.~Blanc, C.~Le~Bris, and F.~Legoll.
\newblock Analysis of a prototypical multiscale method coupling atomistic and
  continuum mechanics.
\newblock {\em M2AN. MathematicalModelling and Numerical Analysis},
  39(4):797--826, 2005.

\bibitem{parisfinitetemp10}
X.~Blanc, C.~Le~Bris, F.~Legoll, and C.~Patz.
\newblock Finite-temperature coarse-graining of one-dimensional models:
  mathematical analysis and computational approaches.
\newblock {\em J. Nonlinear Sci.}, 20:241--275, 2010.

\bibitem{Blanc201384}
X.~Blanc and F.~Legoll.
\newblock A numerical strategy for coarse-graining two-dimensional atomistic
  models at finite temperature: The membrane case.
\newblock {\em Computational Materials Science}, 66(0):84 -- 95, 2013.
\newblock Multiscale simulation of heterogeneous materials and coupling of
  thermodynamic models.

\bibitem{LMDupuy:2005}
L.~M. Dupuy, E.~B. Tadmor, R.~E. Miller, and R.~Phillips.
\newblock Finite temperature quasicontinuum: Molecular dynamics without all the
  atoms.
\newblock {\em Phys. Rev. Lett.}, 95:060202, 2005.

\bibitem{KatoPerturbation}
Tosio Kato.
\newblock {\em Perturbation Theory for Linear Operators}.
\newblock Springer, 1995.

\bibitem{PPMAK:13}
M.~A. Katsoulakis and P.~Plech{\'a}\v{c}.
\newblock {Information-theoretic tools for parametrized coarse-graining of
  non-equilibrium extended systems.}
\newblock {\em J. Chem. Phys.}, 139(7):038332, 2013.

\bibitem{KPS}
M.~A. Katsoulakis, P.~Plech{\'a}\v{c}, and A.~Sopasakis.
\newblock {Error analysis of coarse-graining for stochastic lattice dynamics}.
\newblock {\em SIAM J. Numer. Anal.}, 44(6):2270--2296, 2006.

\bibitem{hyperqc}
W.~K. Kim, M.~Luskin, D.~Perez, E.~B. Tadmor, and A.~F. Voter.
\newblock {Hyper-QC}: An accelerated finite-temperature quasicontinuum method
  using hyperdynamics.
\newblock {\em Journal of the Mechanics and Physics of Solids}, 63:94--112,
  2014.

\bibitem{0951-7715-23-9-006}
Fr\'ed\'eric Legoll and Tony Leli\`evre.
\newblock Effective dynamics using conditional expectations.
\newblock {\em Nonlinearity}, 23(9):2131, 2010.

\bibitem{bqcf.cmame}
Xingjie~Helen Li, Mitchell Luskin, Christoph Ortner, and Alexander~V Shapeev.
\newblock Theory-based benchmarking of the blended force-based quasicontinuum
  method.
\newblock {\em Computer Methods in Applied Mechanics and Engineering},
  268:763--781, 2014.

\bibitem{acta.atc}
M.~Luskin and C.~Ortner.
\newblock Atomistic-to-continuum coupling.
\newblock {\em Acta Numerica}, 22:397--508, 2013.

\bibitem{bqce12}
Mitchell Luskin, Christoph Ortner, and Brian {Van Koten}.
\newblock Formulation and optimization of the energy-based blended
  quasicontinuum method.
\newblock {\em Computer Methods in Applied Mechanics and Engineering},
  253:160--168, 2013.

\bibitem{PUSAV2009}
D.~Perez, B.~P. Uberuaga, Y.~Shin, J.~G. Amar, and A.~F. Voter.
\newblock Accelerated molecular dynamics methods: Introduction and recent
  developments.
\newblock {\em Ann. Rep. Comp. Chem.}, 5, 2009.

\bibitem{PhysRevB.72.144104}
R.~E. Rudd and J.~Q. Broughton.
\newblock Coarse-grained molecular dynamics: Nonlinear finite elements and
  finite temperature.
\newblock {\em Phys. Rev. B}, 72:144104, Oct 2005.

\bibitem{EBTadmor:2013}
E.~B. Tadmor, F.~Legoll, W.~K. Kim, L.~M. Dupuy, and R.~E. Miller.
\newblock Finite-temperature quasicontinuum.
\newblock {\em Appl. Mech. Rev.}, 65, 2013.

\bibitem{EBTadmor:1996}
E.~B. Tadmor, M.~Ortiz, and R.~Phillips.
\newblock Quasicontinuum analysis of defects in solids.
\newblock {\em Philos. Mag. A}, 73:1529--1563, 1996.

\bibitem{GHVineyard:1957}
G.~H. Vineyard.
\newblock Frequency and isotope effects in solid rate processes.
\newblock {\em J. Phys. Chem. Solids}, 3:121--127, 1957.

\bibitem{TSTVoter}
A.~F. Voter and J.~D. Doll.
\newblock Transition state theory description of surface self diffusion:
  Comparison with classical trajectory results.
\newblock {\em J. Chem. Phys.}, 80:5832, 1984.

\end{thebibliography}
\end{document}